\documentclass[a4paper, 11pt, reqno]{amsart}

\usepackage{amsmath, amssymb, amsthm, mathtools, bm, mathrsfs}
\usepackage{hyperref, verbatim, color}
\usepackage[belowfloats]{footmisc}

\usepackage[margin = 1in, bottom = 1in]{geometry}

\usepackage[backend=bibtex,style=alphabetic]{biblatex}
\DeclareFieldFormat[article]{title}{\mkbibemph{#1}}
\DeclareFieldFormat[book]{title}{\mkbibemph{#1}}
\DeclareFieldFormat[thesis]{title}{\mkbibemph{#1}}
\DeclareFieldFormat[online]{title}{\mkbibemph{#1}}

\usepackage{etoolbox, url}
\AtBeginBibliography{\fontsize{9}{11}\selectfont}

\DeclareFieldFormat{url}{\normalfont\url{#1}}
\DeclareFieldFormat{doi}{\normalfont\url{https://doi.org/#1}}
\DeclareFieldFormat{eprint}{\normalfont\url{#1}}

\newtheorem{thm}{Theorem}
\newtheorem{thma}{Theorem}
\newtheorem{thmI}{Theorem}
\newtheorem{thm0}{Theorem}[section]

\newtheorem{lem}[thm0]{Lemma}
\newtheorem{prop}[thm0]{Proposition}
\newtheorem{cor0}[thm0]{Corollary}
\newtheorem{cor}[thma]{Corollary}

\theoremstyle{definition}

\newtheorem{rem}[subsection]{Remark}
\newtheorem*{rems}{Remarks}

\numberwithin{equation}{section}

\newcommand{\floor}[1]{\left\lfloor #1 \right\rfloor}
\newcommand{\ceil}[1]{\left\lceil #1 \right\rceil}

\newcommand{\md}[1]{\, (\text{mod}\ #1)}
\newcommand{\mm}[1]{\, (#1)}

\newcommand{\sg}{\sigma}
\newcommand{\al}{\alpha}
\newcommand{\eps}{\epsilon}
\newcommand{\dl}{\delta}

\newcommand{\tht}{\theta}
\newcommand{\mb}{\mathbb}
\newcommand{\ms}{\mathscr}
\newcommand{\bt}{\beta}
\newcommand{\gm}{\gamma}
\newcommand{\og}{\omega}

\newcommand{\lm}{\limits}
\newcommand{\st}{\substack}
\newcommand{\cg}{\equiv}

\newcommand{\lb}{\lambda}
\newcommand{\sr}{\sqrt}

\newcommand{\fr}{\frac}

\newcommand{\vp}{\varphi}
\newcommand{\lf}{\left}
\newcommand{\rt}{\right}
\newcommand{\sd}{\sideset}
\newcommand{\tps}{\texorpdfstring}
\newcommand{\mc}{\mathcal}
\newcommand{\tl}{\widetilde}
\newcommand{\mf}{\mathsf}

\newcommand{\Lb}{\Lambda}

\newcommand{\tI}{\mathrm{I}}
\newcommand{\tII}{\mathrm{II}}
\newcommand{\res}{\mathop{\mathrm{Res}}}

\renewenvironment{thebibliography}[1]{
  \begin{oldthebibliography}{#1}
    \setlength{\itemsep}{0.07em}
    \setlength{\parskip}{0em}
}
{
  \end{oldthebibliography}
}

\definecolor{dark-red}{rgb}{0.4,0.15,0.15}
\definecolor{dark-green}{rgb}{0.12,0.7,0.18}
\definecolor{dark-blue}{rgb}{0.03,0.07, 1}

\hypersetup{
colorlinks=true,
    pdfborder={0 0 0},
citecolor = cyan,
linkcolor=blue,
}

\begin{document}
\begin{abstract}
We establish completely log-free bounds for exponential sums over the primes and the M\"{o}bius function. Let $0<\eta \leq 1/10$, and suppose $\al = a/q + \dl/x$, with $(a,q)=1$ and $|\dl| \leq x^{1/5 + \eta}/q$, and set $\dl_0 = \max(1, |\dl|/4)$. For $x \geq x_0(\eta)$ sufficiently large, we show that:
\begin{equation*} 
\Biggl| \sum_{n \leq x} \Lb(n) e(n\al) \Biggr| \leq \fr{q}{\vp(q)} 
\fr{\ms{F}_{\eta}\bigl( \fr{\log \dl_0 q}{\log x}, \fr{\log^+ \dl_0/q}{\log x} \bigr) \cdot x }{\sr{\dl_0 q}} \ \text{ and } \ \Biggl| \sum_{n \leq x} \mu(n) e(n\al) \Biggr| \leq \fr{\ms{G}_{\eta}\bigl( \fr{\log \dl_0 q}{\log x}, \fr{\log^+ \dl_0/q}{\log x} \bigr) \cdot x}{\sr{\dl_0 \vp(q)}}, 
\end{equation*}
for all $1 \leq q \leq x^{2/5 - \eta}$, where $\log^+ z = \max(\log z, 0)$, and the functions $\ms{F}_{\eta}$ and $\ms{G}_{\eta}$ are explicitly determined, taking small to moderate values. These bounds improve substantially upon the existing results - particularly with respect to the permissible ranges of $q$, $\dl$ in which log-free bounds are known to hold and potentially with respect to asymptotic functions $\ms{F}_{\eta}$ and $\ms{G}_{\eta}$ as well. Moreover, the range $1 \leq q \leq x^{2/5 - \eta}$ is essentially the best possible we can expect.  
%
The main innovation is a sieve-weighted version of Vaughan's identity  
(Lemma \ref{LFVI}), which 
is effectively 
log-free. We employ several ideas and results from the pioneering work of Helfgott \cite{HH}, 
and particularly, they play a central role in ensuring the log-freeness of the type-I contribution. Also, like in his work, these bounds improve as $\dl$ increases.
\end{abstract}

\title{Log-free bounds on exponential sums over primes}

\author{Priyamvad Srivastav}
\address{Priyamvad Srivastav}
\email{priyamvads@gmail.com}

\date{}

\maketitle

\setcounter{tocdepth}{1}

\tableofcontents

\thispagestyle{empty}

\section{Introduction and results}

Ever since Vinogradov's groundbreaking work (see \cite{Vin37, Vin54}) on exponential sums over primes (and exponential sums in general), the problem of estimating these sums has received considerable attention over the last few decades, enhancing our understanding of the distribution of primes and greatly enriching the field of analytic number theory.

%
Vinogradov studied the sum $\sum_{p \leq x} e(p \al)$, which he decomposed into 
what he called 'type-I' and 'type-II' sums (this predates Vaughan’s identity), and proved an upper bound of the form 
$$\sum_{p \leq x} e(p \al) \ll x\biggl( \sr{\fr{1}{q} + \fr{q}{x}} + e^{-0.5 \sr{\log x}} \biggr) (\log x)^{9/2} ,$$ 
where $\al = a/q + O^*(1/q^2)$, $(a,q)=1$ and $q \leq x$. He also proved
$$\sum_{p \leq x} e(p \al) \ll \fr{x (q \log x)^{5\eps}}{\sr{q} \, \log x}, $$ 
assuming a somewhat different approximation for $\al$. Letting $\eps = 2\log q/(\log \log x)$, leads to a bound $\sum_{n \leq x} \Lb(n) e(n \al) \ll \fr{x \, (\log q)^{10}}{\sr{q}}$, for $q \leq \exp\bigl( \sr{\fr{\log \log x}{10}} \bigr)$, as observed by Ramar\'{e} \cite{Ram1}. 
%

Over the course of time, we have seen several improvements to these results. Let 
$$ S(\al; x) = \sum_{n \leq x} \Lb(n) e(n\al) .$$
Assuming $\al = a/q + O^*(1/q^2)$, Chen \cite{Che1} obtained a bound similar to that of Vinogradov, but with $(\log x)^{3/4} \, (\log \log x)$ replacing
$(\log x)^{9/2}$, while Daboussi \cite{Dab96} obtained the same with $(\log x)^{3/4} \, (\log \log x)^{1/2}$. Later, Daboussi and Rivat \cite{DR} proved a similar bound with fully explicit constants, but with slightly worse factors of $\log x$. 
In \cite{Dab01}, Daboussi had obtained a bound of the form 
$$
S(\al; x) \ll x \sr{\fr{\tau(q) \log^3 q}{\vp(q)}}, \quad  \text{for all} \ \, 1 \leq q \leq \exp((\log x)^{1/3 - \eps}) .
$$
In 1973, a significant simplification was made when Vaughan \cite{RCV} introduced the identity:
\begin{equation}
\label{VV}
\tag{$V$}
\Lb = \mu_{\leq U} * \log \, - \, 1 * \mu_{\leq U} * \Lb_{\leq V} \, + \, 1 * \mu_{> U} * \Lb_{>V} \, + \, \Lb_{\leq V},
\end{equation}
A related identity was discovered by Heath Brown in \cite{HB1}. When $\al = a/q + O^*(1/q^2)$, the Vaughan's identity \eqref{VV} leads to the bound\footnote{Vaughan \cite{RCVb} proves this with a $(\log x)^4$.} (see \cite{IK} for a proof): 
\begin{equation}
\label{VB}
S(\al; x) \ll \biggl( \fr{x}{\sr{q}} + x^{4/5} + \sr{qx}\biggr) (\log x)^3 .
\end{equation}
%
%
%
Also, Chen and Wang \cite{CW} obtained an explicit version of the above bound, with slightly different powers of log. 
%
In \cite{YB}, Buttkewitz, proved a bound using $L$-functions, but it was ineffective. 
An effective bound $\ll \fr{x\, \log \log q}{\sr{q}}$ was proved by Karatsurba \cite{Kar} using $L$-functions. 

A significant achievement was made by Ramar\'{e} \cite{Ram10} (through a refined bilinear decomposition using Bombieri's asymptotic sieve), where he showed that if $\al = a/q + \bt$:
$$
S(\al ; x) \ll \fr{x \sr{q}}{\vp(q)}, \ \ \text{ for } \log q \leq \fr{(\log x)^{1/3}}{50} \quad \text{and} \quad |\bt| \leq \fr{q \exp\bigl( (\log x)^{1/3} \bigr)}{x} .
$$
This bound is completely log-free, more on which will be discussed later in subsection \ref{WLF}. 
Later in \cite{RV1}, Ramar\'{e} and Vishwanadham obtained a fully explicit result of the form\footnote{Their method could yield a better exponent than $1/24$, as they mention.} 
$$|S(a/q ; x)| \leq 1300 \fr{x\sr{q}}{\vp(q)}, \quad \text{ for all } \ 250 \leq q \leq x^{1/24},$$ 
for all $x \geq 1$. They prove similar results for exponential sums over the M\"{o}bius function. 

In 2012, Tao \cite{Tao}, in his work on sum of five primes, considered the smoothed sum 
$S_{\eta}(\al; x) = \sum_n \Lb(n) \eta(n/x) e(n\al)$, and proved an explicit bound similar to \eqref{VB}, having terms of the order $\fr{x \, \log^2 x}{q}$ from type-I and $\fr{x \, \log^2 q}{\sr{q}}$ from type-II, respectively, among other bounds. 

Then, in 2013, Helfgott, in his proof of the ternary Goldbach (see \cite{HH1, HH2, HH3}), brought forth several new ideas, effecting significant improvements on key aspects of the proof, most notably in the minor-arc bounds. The proof is being published in his forthcoming book \cite{HH}, wherein he obtains a fully explicit near log-free bound, asymptotically equivalent to\footnote{His first version(s) had a $\log \dl_0 q$ instead of a $\sr{\log \dl_0 q}$. However, the constants are quite small, which allow him to prove the ternary Goldbach.} (with $\al$ and $\dl_0$ as defined in \eqref{AP} and \eqref{dl0}, respectively):
\begin{equation*}
\label{Hlb}
S_{\eta}(\al; x) \ll \fr{x\, \sr{\log \dl_0 q}}{\sr{\dl_0 \vp(q)}} . 
\end{equation*}
%
We adapt several of his ideas here, and they play a vital role in achieving log-freeness in type-I estimation. We will shed more light 
as we proceed through the relevant sections. 

\smallskip

\subsection{Why log-free bounds}
\label{WLF}
It is natural to ask - \emph{why do we seek log-free bounds}? 

Before answering this fundamental question, we classify log-free bounds into two categories:
\begin{itemize} 
\item \emph{Weakly log-free}: Bounds devoid of powers of $\log x$, but not of $\log q$ (or $\log \log q$), and
\item \emph{Completely log-free}: Bounds devoid of any powers of log (and even $\log \log$),  whatsoever. 
\end{itemize}
In the context of the ternary Goldbach (and related problems), a weakly log-free bound of the type $|S(\al; x)| \leq \fr{cx \, \sr{q}}{\sr{\dl_0} \, \vp(q)} (\log q)^B$, allows one to keep major arcs finite\footnote{This requires the use of a result like \cite[Prop 6.2]{HH3} - and therefore the $\sr{\dl_0}$ saving is vital.} i.e, $\ms{M} = \bigcup\lm_{q \leq C_0} \sd{}{^*}\bigcup\lm_{a \md{q}} B\bigl( \fr{a}{q}, \fr{C_0}{qx} \bigr) $, where $C_0$ is a constant, as opposed to a power of $\log x$. This turns out to be quite advantageous in several situations.  

As we saw before, weakly log-free bounds were already present in the literature, namely due to Vinogradov \cite{Vin54} (for $q \leq \exp(c \sr{\log \log x})$), Daboussi \cite{Dab01} (for $q \leq \exp\bigl( (\log x)^{1/3 - \eps} \bigr)$), Karatsurba \cite{Kar} (for $q \leq \exp(\eps \sr{\log x})$), and that of Helfgott \cite{HH} (similar range for $q$, $ \dl $ that we have), which is near log-free (only $\sr{\log \dl_0 q}$ short) 
and fully explicit. A completely log-free bound was first proved by Ramar\'{e} \cite{Ram1} (for $q \leq \exp\bigl( c' (\log x)^{1/3} \bigr)$), followed by that in Ramar\'{e} - Vishwanadham \cite{RV1} (for $q \leq x^{1/24}$), which was really the first time, the range of $q$ was extended to a power of $x$. 

As observed in \cite{Ram10}, we cannot have a bound better than $x\sr{q}/\vp(q)$ for $S(a/q; x)$, without making progress towards the Siegel zero problem. And if there is a Siegel zero at $1 - \bar{\dl}_1$, then
$$S(a/q; x) \sim \fr{x^{1 - \bar{\dl}_1} \sr{q}}{\vp(q)} .$$ 
%


On a related note, one of the key ingredients for Linnik's theorem (for the least prime in an arithmetic progression) was a log-free zero density estimate for Dirichlet $L$-functions. For an exposition on the importance of log-free zero density estimates for $L$-functions, see 
\cite{AZ}. 

\medskip

The purpose of this paper is to introduce a sieve-weighted version of Vaughan's identity (Lemma \ref{LFVI}), using which we obtain completely log-free bounds for exponential sums over primes and the M\"{o}bius function in a wide range of $q$ and $\dl$ (assuming $\al = a/q + \dl/x$ as in \eqref{AP}). One of these weights is of the Selberg-type, and it makes the type-II estimation quite elegant, even though it adds a layer of complexity to the type-I sums. We also keep track of the implied asymptotic constants (rather functions).
%

\smallskip

\subsection{The main result}
Let 
$$Q \geq x^{2/3} .$$ 
%

By Dirichlet's theorem, there is an approximation:
\begin{equation}
\tag{$*$}
\label{AP}
\al = a/q + \dl/x, \quad \text{with} \ \lf| \dl/x \rt| \leq 1/qQ, \ \ (a,q) = 1 \ \ \text{ and } \ 1 \leq q \leq Q,
\end{equation}
and we set
\begin{equation}
\label{dl0}
\dl_0 = \max(1, |\dl|/4) .
\end{equation}

Further, we shall assume that 
\begin{equation}
\label{eta}
0 < \eta \leq 1/10. 
\end{equation}

\smallskip

The main result is as follows:

\begin{thm}
\label{main}
Let $x \geq x_0(\eta)$ be sufficiently large and $\al$ be as in \eqref{AP} with $Q = x^{4/5 - \eta}$, i.e., $\al = a/q + \dl/x$, $(a,q) = 1$ and $|\dl| \leq x^{1/5 + \eta}/q$. Then, with $\dl_0$ as defined in \eqref{dl0}, we have:
$$
\lf| \sum_{n \leq x} \Lb(n) e(n \al) \rt| \leq  \fr{q}{\vp(q)}  \, \fr{\ms{F}_{\eta}\lf( \fr{\log \dl_0 q}{\log x}, \fr{\log^+ \dl_0/q}{\log x} \rt) \cdot x}{\sr{\dl_0 q}}, \qquad \text{for all } \ 1 \leq q \leq x^{2/5 - \eta},
$$
and\footnote{Note that $1 \leq q \leq x^{2/5 - \eta}$ is equivalent to $1 \leq \dl_0 q \leq x^{2/5 - \eta}$, since $\dl_0 q = \max(q, |\dl|q/4)$ and $|\dl|q/4 \leq x/(4Q) = x^{1/5 + \eta}/4 < x^{2/5 - \eta}$, as $\eta \leq 1/10$.}
$$
\lf| \sum_{n \leq x} \mu(n) e(n\al) \rt| \leq \fr{\ms{G}_{\eta}\lf( \fr{\log \dl_0 q}{\log x}, \fr{\log^+ \dl_0/q}{\log x} \rt) \cdot x}{\sr{\dl_0 \vp(q)}}, \qquad \text{for all } \ 1 \leq q \leq x^{2/5 - \eta},
$$
where 
$\log^+ z$ denotes $\max(\log z, 0)$, and\footnote{We have $\int_A^{B} \sqrt{\fr{t}{t - u}} \, dt =u \log \bigr( \fr{\sr{B} + \sqrt{B-u}}{\sr{A} + \sr{A - u}} \bigr) + \sr{B(B - u)} - \sr{A(A - u)}$, but we leave it in the form of an integral, simply to avoid writing a long expression.} 
\begin{equation} 
\label{FG}
\begin{split}
\ms{F}_{\eta}(u, u_0) &= 1.01 + \fr{14.41}{1 - (\eta +  5u + u_0)/2} \int_{\fr{\eta - \eta^3}{2} + u}^{(2 + \eta + u + u_0)/4}  \sr{\fr{t}{t-u}} \, dt,\\
 \ \text{and } \ \, 
\ms{G}_{\eta}(u, u_0) &= 4.01 \, \fr{1  + \eta^3 -(\eta + 3u + u_0)/2}{ 1 - (\eta + 5u + u_0)/2},
\end{split}
\end{equation}
for all $0 \leq u \leq 2/5 - \eta$ and $0 \leq u_0 \leq \min\bigl(u, 1/5 + \eta \bigr)$\footnote{Clearly, $u_0$ comes into play only when $\dl_0>q$, which happens only if $q<x^{1/10 + \eta/2}$, in which range $\ms{F}_{\eta}$ and $\ms{G}_{\eta} $ are relatively small. Once $q > x^{1/10 + \eta/2}$, $u_0 = 0$ and it's $\ms{F}_{\eta}(u, 0)$ (or $\ms{G}_{\eta}(u, 0)$) we are looking at.}. 
\end{thm}
%
%
\begin{rems} \text{ }
\begin{itemize}
\item The exponent $2/5 - \eta$ in Theorem \ref{main} is the best possible we can expect, using a Vaughan like identity, since in \eqref{VB}, $x^{4/5}$ dominates $x/\sr{q}$ as soon as $q \geq x^{2/5}$. So, one cannot expect a bound $S(\al ; x) \ll x \sr{q}/\vp(q)$, when $q>x^{2/5}$ via an approach of this type.

\vskip 0.025in

\item These bounds hold even at $q = 1$ or $2$! If $\al = 1/2$, then $\sum_{n \leq x} \Lb(n) e(n/2) \sim -x$, whereas Theorem \ref{main} gives a bound $ \approx 8.25 x$. So, the sufficiently large is only for $x$, and not for $q$, $\dl_0q$. 

\vskip 0.025in 

\item The constant $x_0(\eta)$ is effectively computable, depending only on $\eta$.
%
\end{itemize}
\end{rems}
\smallskip
%
%
%
%
Choosing $\eta =1/15$ in Theorem \ref{main}, we 
can deduce:

\begin{cor}
\label{maincor}
Let $\al = a/q + \dl/x$, with $(a, q)=1$ and $|\dl| \leq  x^{4/15}/q$. Then, for all $x \geq x_0$:
\begin{equation*}
\label{115b}
\lf| \sum_{n \leq x} \Lb(n) e(n\al) \rt| \leq \fr{51}{\sr{\dl_0}} \, \fr{x\sr{q}}{\vp(q)} \quad \text{and} \quad \lf| \sum_{n \leq x} \mu(n) e(n\al) \rt| \leq \fr{15 \, x}{\sr{\dl_0 \vp(q)}}, \quad \text{for all } \ 1 \leq q \leq x^{1/3} .
\end{equation*}
\end{cor}
\begin{proof}
Let $\eta = 1/15$, and assume that\footnote{We include a proof in Markdown of the ancillary file Calculations.ipynb.} $\ms{F}_\eta(u,0)$, $\ms{G}_{\eta}(u,0)$ as well as $\ms{F}_\eta(u,u)$, $\ms{G}_{\eta}(u,u)$ are both increasing in $u$. Also, it is seen that $\ms{F}_{\eta}(u, u_0)$, $\ms{G}_{\eta}(u, u_0)$ increase with $u_0$ as $u$ remains fixed.  

Hence, by Theorem \ref{main}:
$$
\Bigl| \sum_{n \leq x} \Lb(n) e(n\al) \Bigr| \leq \fr{1}{\sr{\dl_0}}\fr{x\sr{q}}{\vp(q)} \, \max\biggl( \ms{F}_{\eta}(1/5 + \eta, 1/5 + \eta), \, \ms{F}_{\eta}(2/5 - \eta, 0) \biggr),
$$ 
where this maximum equals\footnote{Again, evaluated in the ancillary file Calculations.ipynb.} $50.97 \ldots$ when $\eta = 1/15$. Doing the same for $\ms{G}_{\eta}$, gives $14.04 \ldots$. 
\end{proof}

\smallskip

\subsubsection{Comparison with previous results}
In \cite{Ram10}, the bound was effective (explicit constants not mentioned), and true for $q \leq \exp\bigl( \fr{(\log x)^{1/3}}{50} \bigr)$, $|\dl| \leq q \exp\bigl( (\log x)^{1/3} \bigr)$, while in \cite{RV1}, it was a fully explicit result for $q \leq x^{1/24}$, $\dl = 0$, but with a large constant. 

Helfgott's \cite{HH} bounds are weakly (though nearly) log-free, fully explicit with good constants, and valid for a wide range of $q$, $\dl$, similar to what we have. They are good enough for the ternary Goldbach, combined with the finite verification of GRH by Platt \cite{Pla}.

What we have in Theorem \ref{main}, is a semi-explicit result (i.e., we keep track of the asymptotic constants, but not the $x_0$, though it is effectively computable),  true for $1 \leq q \leq x^{2/5 - \eta}$ and $|\dl| \leq x^{1/5 + \eta}/q$, surpassing the earlier known range(s) ($q \leq x^{1/24}$, $\dl = 0$ in \cite{RV1}) by a substantial margin. The flexibility of the method allows us to go upto the exponent $2/5 - \eta$, which is essentially the best possible via such an approach. 
Also, as seen in Corollary \ref{maincor}, the asymptotic constants are not too large. 

The approach here should also lead to fully explicit results within the same range of $q$ and $\dl$ as in Theorem \ref{main}, but with larger constants. We keep things this way, primarily for simplicity.  

\bigskip

\section*{Notation}
Throughout the paper, we assume $\al$ to have an approximation \eqref{AP} and that fixes $q$, $\dl$ and $\dl_0$ for us. However, in section \ref{TIS}, we work with an approximation \eqref{AP0} for $\al$, where $q_0$ and $y$ are used instead of $q$ and $x$. This is mainly done to preserve a sense of generality, as these results are later used in different settings. 

The parameter $\eta \in (0, 1/10]$ in \eqref{eta} is fixed. We determine $1 < U < U_1$ through the weights $\tht(d)$, $\tht'(d)$ in \eqref{tht} and $R>1$ from the weights $\{\lb(d)\}$ in \eqref{lbR}. The parameter $V>1$ occurs in Lemma \ref{LFVI}, and all of $U$, $U_1$, $R$, $V$ are at our disposal, subject to certain constraints. They are assumed to be sufficiently large and tend to $\infty$ as $x \to \infty$. We also choose $Q = x^{4/5 - \eta}$. Also, $\eps>0$ is assumed to be sufficiently small. In intermediate steps, whenever $\eps$ makes an appearance, we implicitly assume that the 'sufficiently large' also depends upon $\eps$, although none of the final bounds really depend upon $\eps$. In section \ref{TIS}, we often use the parameter $\rho \in (0, 1/8]$ instead of $\eta$. 

We use $l$, $\ell$, $m$, $n$, $d$, $r$, $k$, $v$ to denote positive integers. As usual, $p$ is reserved for a prime. For positive integers $m$ and $n$, we shall denote their GCD and LCM by $(m,n)$ and $[m,n]$, respectively. Also, we write $d \mid v^{\infty}$, to imply that 
$d \mid v^N$, for large enough $N$. Moreover, $\Lb$ denotes the Von-Mangoldt function, $\mu$ the M\"{o}bius function, $\tau$ the divisor function, $\vp$ the Euler totient function, $c_r(n)$ the Ramanujan sum and the function $h$ is defined in \eqref{h}. Also, we write $\sd{}{^*} \sum_{a \md{q}}$ to denote that $(a,q) = 1$. We find it convenient to use the shorthand $m \cg a \mm{q}$ instead of $m \cg a \md{q}$. The symbol $*$ shall denote the Dirichlet convolution. 

Moreover, $\|z\|$ denotes the distance of $z$ from the nearest integer, $\{z\}$ denotes the fractional part of $z$ and $\psi(x) = \{x\} - 1/2$. The function $\log^+(z)$ denotes $\max(\log z, 0)$. In general, $z^+$ will denote $\max(z, 0)$. Also, as is the usual practice, we denote $e(x): = e^{2 \pi i x}$. The support of $f$ is denoted by $\text{supp}(f)$, while $f_{\leq V}$ or $f_{>V}$ denote the function $f$ with support restricted to the specified range. For a set $S$, we shall often denote its cardinality by $|S|$. Moreover, $\bm{1}_{\mc{P}}$ denotes the indicator function of whether a given property $\mc{P}$ holds. We occasionally use $n \sim N$ to denote that $N < n \leq 2N$.  
 
The '$o$' and '$O$' notations are as usual, and the $O$-constants may often depend on $k$ or $\eta$, or some such fixed and bounded quantities. The '$\ll$' will be a substitute for $O$ at several places. The symbol $O^*$ means that the implied $O$-constant is $1$, i.e., $F = O^*(G) \iff |F| \leq |G|$. 

All other notation is described in the relevant sections. 

\bigskip

\section{A sieve-weighted Vaughan's identity}
As pointed out in \cite{HH}, Vaughan's identity \eqref{VV} is not \emph{log-free}, in the sense that if one takes trivial bounds over each of its individual components, it results in a bound proportional to $x \, \log^2 x$, which is worse than the trivial bound $O(x)$ for the entire sum.   

So, Helfgott uses a smoothed version of Vaughan’s identity (also appears in the work of Sedunova \cite[Lemma 1]{AS}), where the sharp truncations $\mu_{\leq U}$ and $\mu_{>U}$ in \eqref{VV} are replaced by $\mu \phi$ and $\mu (1 - \phi)$, respectively, where $\phi(d) = 1$, for $d \leq U$, $\phi(d) = \fr{\log U_1/d}{\log U_1/U}$, for $U < d \leq U_1$ and $0$ elsewhere. In other words, $\mu \phi$ equals the Barban-Vehov weights introduced in \cite{BV}. The advantage is that it leads to a $\sr{\log}$ saving in the type-II sums, owing to the $L^2$-norm of $1 * \mu \phi$ (see \cite{Gr2}). Their approach is ultimately just $\sr{\log}$ short\footnote{His first version(s) on arXiv did not incorporate this particular idea, so they were short by a $\log$.} of a completely log-free bound. 

\medskip

Inspired by this smooth version of Vaughan's, we derive a new identity by inviting weights of the Selberg-type\footnote{In principle, any reasonable sieve weights could work, but the Selberg weights make matters quite elegant, owing to its properties in Lemma \ref{lbcr} and Lemma \ref{lbsum} (a).} into the fold (in addition to the Barban-Vehov weights). This yields a factorised component in the type-II term, ultimately leading to completely log-free estimates for both the type-I and type-II sums. It however adds a layer of complexity to the type-I sums. 

\smallskip

Let $1<U<U_1$ and $R > 1$ be parameters to be chosen later, and
\begin{itemize}
\item[(i)] Let $\{\lambda(d)\}$ be the Selberg sieve type weights supported on $d \leq R$ and $(d,q) = 1$, i.e., 
\begin{equation} 
\label{lbR}
\lambda(d) = \fr{d \mu(d)}{\vp(d)} \fr{G_{qd}(R/d)}{G_q(R)} \bm{1}_{(d,q) = 1},
\end{equation}
where 
$$
G_{\ell}(x) = \sum\lm_{\st{r \leq x \\ (r,\ell) = 1}} \fr{\mu^2(r)}{\vp(r)} . 
$$

\medskip

\item[(ii)] Define weights $\{\theta(d)\}$ and $\{\theta'(d)\}$ of the Barban-Vehov type by:

\begin{equation}
\label{tht}
\tht(d) = \mu(d) \begin{cases} 0, & d \leq U, \\ \frac{\log d/U}{\log U_1/U}, & U < d \leq U_1, \\ 1, & d>U_1, \end{cases} \quad \, \text{and} \, \quad \tht'(d) = \mu(d) \begin{cases} 1, & d \leq U, \\ \frac{\log U_1/d}{\log U_1/U}, & U < d \leq U_1, \\ 0, & d>U_1 . \end{cases}
\end{equation}
Then clearly, $\tht + \tht' = \mu$. 
\end{itemize}

\smallskip

Now, let
\begin{equation}
\label{h}
h(d) := \sum_{[d_1, d_2] = d} \lambda(d_1) \, \theta'(d_2).
\end{equation}

It is clear that $h$ is supported on $[1, U_1 R]$ and that $1 * h$ factorizes as 
\begin{equation*}
 1 * h = (1 * \theta') (1 * \lambda)  .
\end{equation*}

\smallskip

\begin{lem}[Sieve-weighted Vaughan's identity]
\label{LFVI}
Let $1 < U < U_1$, $R > 1$ be as before and $h$ be as defined in \eqref{h}. Then, for any $V>1$, we have
\begin{equation}
\tag{$V_{\Lb}^*$}
\label{VI*}
\Lambda = h * \log \, -  \, 1 * h * \Lambda_{\leq V} \, + \, (1 * \theta) (1 * \lambda) * \Lambda_{>V} \, + \, \Lambda_{\leq V},
\end{equation}
\begin{equation*}
\label{Vmu*}
\tag{$V_{\mu}^*$}
\mu = h \, - \, 1 * h * \mu_{\leq V} + (1*\tht)(1 * \lb)*\mu_{>V} \, + \, \mu_{\leq V}.
\end{equation*}
\end{lem}
\begin{proof}

We begin with the observation that:
\begin{equation*}
\label{Iid}
I = 1 * \mu = (1 * \mu) (1 * \lb) = (1 * \tht)(1 * \lb) +  (1 * \tht')(1 * \lb). 
\end{equation*}

For $\Lb$, as usual express $\Lambda = \Lambda_{\leq V} + \Lambda_{>V}$ and note that:
\begin{equation*}
\begin{split}
\Lb_{>V} &= \Lb_{>V} * (1 * \tht)(1 * \lb) \, + \, \Lb_{>V} * (1 * \tht')(1 * \lb) \\
&= \Lb_{>V} * (1 * \tht)(1 * \lb) \, + \, \lf( \Lb - \Lb_{\leq V} \rt) * (1 * h) \\
&= h * \log \, -  \, 1 * h * \Lb_{\leq V} \, + \,  (1 * \tht)(1 * \lb) * \Lb_{>V}. 
\end{split}
\end{equation*}

Similarly for $\mu$, we follow the same procedure and obtain:
\begin{equation*}
\begin{split}
\mu_{>V} &= \mu_{>V} * (1 * \tht)(1 * \lb) \, + \, \lf( \mu - \mu_{\leq V} \rt) * (1 * h) \\
&= h \, -  \, 1 * h * \mu_{\leq V} \, + \,  (1 * \tht)(1 * \lb) * \mu_{>V}. 
\end{split}
\end{equation*}

This completes the proof. 
\end{proof}

\smallskip

\begin{rems}\text{ }
\begin{itemize}
\item Lemma \ref{LFVI} holds for any weights $\{\lb(d)\}$ satisfying $\lb(1) = 1$ and $\{\tht'(d)\}$, $\{\tht(d)\}$ satisfying $\tht + \tht' = \mu$. As such, we do not require these weights to be of any specific type. Moreover, if we let $U_1 = U$, and $R = 1$, this simply reduces to the original Vaughan's identity. 

\vskip 0.075in

\item Overall, this makes the type-I sum more complicated (owing to the appearance of the function $h$), but can be dealt with, through a careful analysis. Also, if $U_1 = UR_1$, then the main type-I contribution (minus the log-factors) is proportional to $x/q + UV(RR_1)$, as opposed to $x/q + UV$ in the usual Vaughan's identity. However, 
this allows for a critical saving of $\sr{\log R \cdot \log R_1}$ in the type-II contribution. 

\vskip 0.075in

\item There is a striking resemblance between the factor $(1 * \tht) (1 * \lb)$ that appears above and the mollifier, or zero detector (multiplied by a character and possibly a smoothing) that appears when obtaining log-free zero density estimates for Dirichlet $L$-functions (see Graham \cite{Gr1}, Heath-Brown \cite{HB} or Motohashi \cite{YM}). There too, the Barban-Vehov weights and Selberg-type weights combine in an identical fashion. Hence, a decomposition of the form \eqref{VI*} or \eqref{Vmu*} does makes sense. 

\vskip 0.075in

\item In \cite{RV1} (also \cite{RV2}), they consider a family of bilinear decompositions for $\Lb$ and $\mu$ (in fact for a class of functions), for each $r \leq R$ (where the Ramanujan sums $c_r(n)$ as well as Barban-Vehov weights come into play). Using these identities, they obtained log-free bounds on exponential sums over primes and various arithmetic functions, but in a restricted range. The Ramanujan sums are implicit in the identity we use, owing to the connection between Selberg sieve weights and the Ramanujan sums (Lemma \ref{lbcr}). 
\end{itemize}
\end{rems}
\medskip

\section{Outline of the proof}
\label{OTP}
Let us lay out the basic strategy of the proof. 
The type-I estimation relies significantly on the results of Helfgott \cite[Chapter 11]{HH}, and we follow a path similar to his, although adapting his method directly to our setup is not quite straightforward\footnote{One of the difficulties in directly applying Helfgott's results, is that they appear with a smoothing, while we have a sharp truncation, meaning that many of the results that exploit smoothing are no more applicable. Moreover, the presence of the function $h$ instead of $\mu$ or $\mu \phi$ causes some additional complexity.}. While in the type-II, we largely pursue a different approach, bringing in some ideas from the author's thesis \cite[Chapter 3]{PS} and making efficient use of the Selberg sieve weights. 

\subsubsection*{Type-I sums}
Let us first briefly discuss the main ideas in the type-I sums. 
For simplicity, consider the sum (akin to Theorem \ref{TI} (a) or the first term of \eqref{VI*}) %
$$
S_{\tI, 1}(\al; x) = \sum_m h(m) \sum_{mn \leq x} \log n \, e(mn \al).  
$$ 
The starting idea of Helfgott, 
which we adopt, is to separate $m$ for which $q \mid m$ from $q \nmid m$. This is particularly useful, and what it does is: (i) when $q \mid m$, $e(mn\al)$ equals $e(mn\dl/x)$, and this is dealt via Proposition \ref{hqm} (which plays a similar role to Lemmas 11.6 and 11.8 of \cite{HH}), and (ii) when $q \nmid m$, the trigonometric sums have better bounds (see \cite[Lemma 11.1]{HH}), since it is mainly the $m$'s divisible by $q$ which create difficulty (because $\al \approx a/q$!). 

For those $m$ with $q \mid m$ (call this sum $S_{\tI, 1}^{'}$), we apply Proposition \ref{hqm} with $k=1$ to get:
$$
\lf| S_{\tI,1}^{'}(\al ; x) \rt| \leq \fr{x}{\dl_0} \lf( \biggl| \sum_{m \cg 0 \mm{q}} \fr{h(m)}{m} \log \fr{x}{m} \biggr| + \log x \, \biggl| \sum_{m \cg 0 \mm{q}} \fr{h(m)}{m} \biggr| \rt) \, + \, O\lf( \log x \sum_{m \cg 0 \mm{q}} |h(m)| \rt) .
$$
Next, Lemma \ref{qhm} is invoked to bound the above sum totally. Moreover, it is in Lemma \ref{qhm}, that the Selberg weights $\{\lb(d)\}$ greatly simplify matters (in addition to the type-II), which is one of the most critical parts of the type-I estimation. 
Also, Lemma \ref{qhm} implicitly uses certain results of Ramar\'{e} \cite{Ram1} on partial sums of the M\"{o}bius function.    

For those $m$ for which $q \nmid m$ (call this sum $S_{\tI, 1}^{''}$), we invoke Lemma \ref{sumdD} (consequence of \cite[Lemma 11.1]{HH} and the alternate approximation \cite[Lemma 2.2]{HH}), but not directly. Here, we need to handle the $h(m)$ a bit delicately, as using the bound $|h(m)| \leq 3^{\nu(m)}$ would not rank amongst the greatest of ideas (unless one fancies an unwelcome $\log^2$-factor). We get around this by writing $m = gd_1d_2$, with $(d_1, d_2) = (gd_1, q) = (gd_2, q) = 1$ and $gd_1, gd_2 \leq R$. This leads to (see eq. \eqref{t3mp1}):
\begin{equation*}
\lf| S_{\tI, 1}^{''}(\al; x) \rt| 
\leq \fr{1}{\log \fr{U_1}{U}} \sum\lm_{\st{g \leq R \\ (g, q) = 1}} \sum\lm_{\st{d_1 \leq R/g \\ (d_1, gq) = 1}} |\lb(gd_1)| \sum\lm_{\st{ d_2 \leq U_1/g \\ (d_2, g) = 1 \\ q_0 \nmid d_2}} \log \fr{U_1/g}{d_2} \lf| \sum\lm_{n \leq \fr{y}{gd_1d_2}} \log n \ e(n d_2 (gd_1 \al)) \rt| . 
\end{equation*}
and then apply Lemma \ref{sumdD} to the inner sum over $d_2$ (with $\al$ replaced by $gd_1 l \al$, and other parameters altered likewise\footnote{This leads to certain constraints on the parameters $U$, $U_1$, $V$ and $R$.}), followed by an application of Lemma \ref{lbsum} (c). It is worth noting that the alternate approximation\footnote{It seems, we could avoid the use of an alternate approximation, but at the cost of restricting the range of $q$ in Theorem \ref{main} to $q \leq x^{1/3 - \eta}$.} comes in use only when $m$ is large. 

Naturally, the longer type-I sums (Theorem \ref{TI} (b), (c)) are somewhat more complex, as they have an extra variable $\ell$, but the same process essentially applies (also see \cite[Section 11.5]{HH}).
\subsubsection*{Type-II sums}
Let us now go through the main ideas in the type-II sums. As usual, we reduce these sums over $mn \leq x$ into dyadic intervals $\mc{M} \subseteq (M, 2M]$ and $\mc{N} \subseteq (N, 2N]$, and assume also for simplicity that $MN \asymp x$. Note that we have $M  > V$ and $N > U$ (since $1 * \tht$ vanishes otherwise). Further, the sum over $\Lb(m)$ is essentially a sum over primes, with a small error. So, let
$$
S(\al, \mc{M}, \mc{N})  = \sum_{p \in \mc{M}} \log p \sum_{n \in \mc{N}} (1 * \tht)(n) \, (1 * \lb)(n) \, e(n p \al) .
$$
The first step is the Cauchy-Schwarz, which gives:
$$
|S(\al, \mc{M}, \mc{N})|^2 \leq (1 + \eps) M \log M  \sum_{p \in \mc{M}} \lf|\sum_{n \in \mc{N}} (1 * \tht)(n) \, (1 * \lb)(n) \, e(n p \al)\rt|^2. 
$$
Now, let $S_{21}(p\al, \mc{N}) = \sum_{n \in \mc{N}} (1 * \tht)(n) \, (1 * \lb)(n) \, e(n p \al)$. The first key step is to decompose $(1*\lb)(n)$ 
via Lemma \ref{lbcr}, which leads to
\begin{equation*}
\begin{split}
G_q(R) S_{21}(p\al, \mc{N}) &= \sum\lm_{\st{r \leq R \\ (r, q) = 1}} \fr{\mu(r)}{\vp(r)} \sum_{n \in \mc{N}} (1 * \tht)(n)  \, c_r(n) \, e( n p \al)\\
 &= \sum_{\st{r \leq R \\ (r, q) = 1}} \fr{\mu(r)}{\vp(r)} 
 \sd{}{^*}\sum_{a' \md{r}} \sum_{n \in \mc{N}} (1 * \tht)(n) \, e\bigl( n(p \al + a'/r) \bigr).
 \end{split}
\end{equation*}
%
Now, again apply Cauchy-Schwarz, simplify\footnote{After Cauchy-Schwarz, there will be a $G_q(R)^2$ in the LHS, and a $G_q(R)$ on the RHS, one of which cancels.}, and then sum over $p \in \mc{M}$, to get:
\begin{equation*}
G_q(R) \sum_{p \in \mc{M}}| S_{21}(p\al, \mc{N})|^2 \leq \sum_{p \in \mc{M}} \sum_{\st{r \leq R \\ (r, q) = 1}} \sd{}{^*}\sum_{a' \md{r}} \biggl| \sum_{n \in \mc{N}} (1 * \tht)(n) \, e\bigl( n(p \al + a'/r) \bigr) \biggr|^2 .
\end{equation*}
What we see above is similar to what arises after applying the Montgomery's inequality (see \cite[Lemma 9.7]{IK}) and summing over both $r$ and $p$. This appears naturally in the large sieve for primes, also used in the proof of the Brun-Titchmarsh theorem \cite{MV73}. In other words, the presence of the $(1*\lb)$ factor mimics the effect of an averaged Montgomery's inequality. 

Now, the standard arguments apply. One needs to ensure that $\al_{p, a'/r} = p\al + a'/r$ are well-spaced in $\mb{R}/\mb{Z}$. 
Here, we also use a combinatorial lemma (Lemma \ref{BM}) from the author's thesis, i.e., partitioning $\mc{M}$ into at most $\fr{2M}{\dl_0 \vp(q) \, \log \fr{M}{\dl_0 q}}$ sets, to ensure that $\al_{p, a'/r}$'s are well-spaced. This allows us to save a factor $\sr{\fr{\vp(q)}{q} \, \log \fr{M}{\dl_0q}}$. The large sieve inequality now finally gives\footnote{There are three things going on: (i) the $(1*\lb)$ factor saves a $\sr{G_q(R)}$ by mimicking an averaged Montgomery's inequality or a large sieve for primes, (ii) the $L^2$-norm of $(1 * \tht)$ saves a $\sr{\log U_1/U}$, and (iii) the partition of $\mc{M}$ via Lemma \ref{BM} (based on the Brun-Titchmarsh) further saves $\sr{\fr{\vp(q)}{q} \log \fr{M}{\dl_0q}}$.}:
\begin{equation*}
|S(\al, \mc{M}, \mc{N})| \leq \fr{q}{\vp(q)} \fr{C \, x}{\sr{\dl_0 q}} \sr{\fr{\log M}{\log R \, \log \fr{M}{\dl_0 q} \, \log \fr{U_1}{U}}}.
\end{equation*}
After summing over the dyadic intervals\footnote{This gives rise to another log-factor.}, we obtain the desired log-free bound. Here, the $\log R$ comes from $G_q(R) \geq \fr{\vp(q)}{q} \log R$ (by \cite[eq 1.3]{LR}), while  $\log U_1/U$ comes from the $L^2$-norm of $1*\tht$, for which an asymptotic formula was proved by Graham \cite{Gr2}.  

Actually, the sole purpose of deriving an identity of the type \eqref{VI*} (or \eqref{Vmu*}), was that it's type-II term has a $(1 * \tht)(1 * \lb)$, which is tailored for the application of all these ideas simultaneously.

\bigskip

\section{Type-I sums}
\label{TIS}
Let us now begin with the type-I sums. Several bounds for trigonometric sums will come into play, where, apart from using the results of Helfgott \cite{HH}, we also revisit the work of Daboussi-Rivat \cite{DR}. Most of the key arguments that we use here, are summarised in Section \ref{OTP}, and it might be helpful to have a quick look at it.  

\smallskip

Let $f$ be an arithmetic function. Define the type-I sums as follows:
\begin{equation}
\begin{split}
\label{tIsumdef}
S_{\tI, 1}(\al; x) &= \sum_{m} h(m) \sum_{mn \leq x} (\log n) \, e(m n \al), \\
S_{\tI, 2, f}(\al; x) &= \sum_{\ell \leq V} f(\ell) \sum_{m} h(m) \sum_{mn \leq x/\ell} e(\ell m n \al) .
\end{split}
\end{equation}

\smallskip

Here is the main result of this section:
\begin{thmI}
\label{TI}
Assume that the following conditions hold: 
\begin{equation}
\label{CI}
\tag{$\mc{C}1$}
\begin{split}
U \geq 9\dl_0(Rq)^{1 + \eta/2}, \quad V \geq x^{\eta/3} \dl_0 q, \quad R \geq x^{\eta/4},  \quad \ U_1 V R \leq  \fr{x }{8 \dl_0}, \quad \ qVR \leq Q. 
\end{split}
\end{equation}
Note that the conditions in \eqref{CI} clearly imply $U_1R \leq \fr{x^{1 - \eta/3}}{8 \dl_0^2 q}$.
%

\vskip 0.02in

Let $\al$ be as in \eqref{AP} and $x \geq x_0(\eta)$ be sufficiently large. Then:
\begin{itemize}
\item[(a)]
\begin{equation*}
\label{T11}
|S_{\tI, 1}(\al; x)|
\leq  \fr{x}{\dl_0 \vp(q)} \lf( 1 + O\biggl( \fr{\eta^{-1} \log x}{\log qR \,  \log \fr{U_1}{U}  } \biggr)\rt) .
\end{equation*}
%
\item[(b)]
\begin{equation*}
\label{T12}
|S_{\tI, 2, \Lb}(\al; x)|
\leq \fr{3 \, x \, \log Vq \, }{\dl_0 \vp(q) \, \log \fr{U_1}{U} \log \fr{U}{qR}} \, + \, 
O\lf( \fr{q}{\vp(q)}\fr{U_1 V R \, \log^2 x}{\log R \, \log \fr{U_1}{U}} \rt) .
\end{equation*}
\item[(c)]
\begin{equation*}
\label{T13}
|S_{\tI, 2, \mu}(\al; x)| \leq \fr{3 \, x \, \tau(q) \, \log V}{\dl_0 \vp(q) \, \log \fr{U_1}{U} \, \log \fr{U}{qR}} \, + \, O\lf(\fr{q}{\vp(q)} \fr{U_1 V R \, \log^2 x}{\log R \, \log \fr{U_1}{U}} \rt) .
\end{equation*}
\end{itemize}
\end{thmI}
%
%

\medskip

\subsection{Preliminary results}
To preserve a general setup, we shall, in this section,  assume:
\begin{equation}
\label{AP0}
\tag{$**$}
\al = a_0/q_0 + \dl/y, \ \ (a_0, q_0) = 1, \ \ |\dl|/y \leq 1/q_0 Q_0 \quad \text{and} \quad q_0  \leq Q_0 .
\end{equation}

Here $y > 1$ and as before, we let $\dl_0 = \max(1, |\dl|/4)$. 

\smallskip
We begin with the following lemma:
\begin{lem}
\label{VdC}
Let $\phi$ be increasing and differentiable on $[a,b]$ and let $|\bt| \leq 1/2$. Then
$$
\sum_{a < n \leq b} \phi(n) e(n \bt) = \int_a^b \phi(t) \, e(t \bt) \, dt \,  + \, O(|\phi(b)|). 
$$
\end{lem}
\begin{proof}
We follow the proof of the standard van der Corput result. First, by Euler-Maclaurin summation (letting $\psi(t) = \{t\} - 1/2$), we have
\begin{equation*}
\begin{split}
\sum_{a < n \leq b} e(n \bt) &= \int_a^b e(t \bt) \, dt  + 2\pi i \bt \int_a^b \psi(t) \, e(t \bt) \, dt  +O(1).
\end{split}
\end{equation*}
Now,
\begin{equation*}
\begin{split}
\int_a^b \psi(t) e(t \bt) \, dt 
&= -\fr{1}{\pi}\sum_{v = 1}^{\infty} \fr{1}{v} \int_a^b  \sin 2 \pi v t  \,  e(t \bt) \, dt 
= \fr{1}{2 \pi i}\sum_{v = 1}^{\infty} \fr{1}{v} \int_a^b 
\bigl( e((\bt - v)t) - e((\bt + v)t) \bigr)\, dt \\
&= \fr{1}{2 \pi i}\sum_{v = 1}^{\infty} \fr{1}{v}
\lf( \fr{1}{\bt - v} \int_a^b  d\bigl( e((\bt - v)t) \bigr)  -  \fr{1}{\bt + v}\int_a^b  d\bigl( e((\bt + v)t) \bigr) \rt)
= O(1). 
\end{split}
\end{equation*}
The interchange of the $v$-sum with the integral is justified by the Dominated convergence theorem, as the sums $\sum_{v=1}^N \fr{\sin 2 \pi v t}{v}$ are uniformly bounded for $t \notin \mb{Z}$ (see also \cite[Ex 3, Ch 4]{IK}).

Now, by partial summation (letting $S(u) = \sum_{a<n \leq u} e(n\bt)$)
\begin{equation*}
\begin{split}
\sum_{a < n \leq b} \phi(n) e(n \bt) &= \phi(u) S(u) \big\vert_{a}^{b} - \int_a^b \phi'(u) S(u) \, du \\
&= \phi(b) \lf( \int_a^b e(t\bt) \, dt + O(1) \rt) - \int_a^b \phi'(u) \lf( \int_a^u e(t \bt) \, dt  + O(1) \rt) \, du\\
&= \int_a^b \phi(t) e(t \bt) \, dt \, + \, O(|\phi(b)|). 
\end{split}
\end{equation*}
\end{proof}

The next Proposition plays the same role as Lemmas 11.6 and 11.8 of \cite{HH}. We however, can't use the Poisson summation owing to the sharp truncation at $y$. So, we use Lemma \ref{VdC} to convert the sum into an integral, and after a change of variable, further transform the inner sum into a complex integral, with the residue yielding us the main term. 
\begin{prop}
\label{hqm}
Let $\al$ be as in \eqref{AP0} and let $f: \mb{N} \to \mb{C}$ be supported within $[1, \fr{y}{8 \dl_0} ]$. Then, for all $k \geq 0$:
\begin{equation*}
\begin{split}
\lf| \sum\lm_{\st{m \cg 0 \mm{q_0}}} f(m) \sum_{n \leq y/m} (\log n)^k \, e(mn\al)  \rt| &\leq \fr{y}{\dl_0} \sum_{l=0}^k \binom{k}{l} (\log y)^{l}  \lf| \sum_{m \cg 0 \mm{q_0}} \fr{f(m)}{m} \Bigl( \log \fr{y}{m} \Bigr)^{k-l} \rt| \\ 
&\quad  +  O\lf( (\log y)^{\rho_k} \sum_{m \cg 0 \mm{q_0}} |f(m)| \rt),
\end{split}
\end{equation*}
where $\rho_k = \begin{cases} 2, & k = 0, \\ k, & k \geq 1, \end{cases}$ and the $O$-constant depends only on $k$.
\end{prop}
\begin{proof}
Let us denote the given sum by $\tl{S_k}(\al, y)$. Note that $|m\dl/y| \leq \fr{y}{8\dl_0} \cdot \fr{|\dl|}{y} \leq 1/2$. 
Hence
\begin{equation}
\begin{split}
\label{tmpI1}
\tl{S_k}(\al, y) 
&= \sum_{m \cg 0 \mm{q_0}} f(m) \sum_{mn \leq y} (\log n)^k \, e\left(mn(a_0/q_0 + \dl/y)\right)
 = \sum_{m \cg 0 \mm{q_0}} f(m) \sum_{mn \leq y} (\log n)^k \, e\lf(mn\dl/y\rt)\\
&= \sum_{m \cg 0 \mm{q_0}} f(m) \biggl( \int_{1}^{\fr{y}{m}} 
 (\log t)^k e\lf(\fr{m \dl t}{y}\rt) \, dt \, + \, O\lf(\log^k \fr{y}{m} \rt)  \biggr)\\
&= y \, \int_{\fr{1}{y}}^1 \lf( \sum\lm_{\st{m \leq yt\\ m \cg 0 \mm{q_0}}} \fr{f(m)}{m} \log^k \fr{yt}{m} \rt) \, e(\dl t) \, dt \, + \, O\lf( (\log y)^k \sum_{m \cg 0 \mm{q_0}} |f(m)|\rt),
\end{split}
\end{equation}
%
%
where we apply Lemma \ref{VdC} in the second line and $t \mapsto yt/m$ in the third line. 

\smallskip

Let $c > 0$ and let $\mc{F}_0(s) = \sum_{m \cg 0 \mm{q_0}} \fr{f(m)}{m^{s}}$, which is a finite sum. The main term in \eqref{tmpI1} is:
\begin{equation}
\label{tmpI2}
\begin{split}
&\quad y \int_{\fr{1}{y}}^1 \lf(\fr{k!}{2 \pi i} \int_{c - i T}^{c + i T} \mc{F}_{0} (1 + s) \fr{(yt)^s}{s^{k+1}} \, ds  + O\biggl( \fr{(yt)^c}{T^k}\sum_{m \cg 0 \mm{q_0}} \fr{|f(m)|}{m^{1+c} \lf( 1 + T\bigl| \log \fr{yt}{m} \bigr| \rt) } \biggr) \rt) \ e(\dl t) \, dt\\
&=\fr{k! \, y}{2 \pi i} \int_{c - i T}^{c + i T} \mc{F}_{0}(1 + s) \fr{y^s}{s^{k+1}} \biggl(\int_{\fr{1}{y}}^1 t^s e(\dl t) \, dt \biggr) \, ds + O\lf( \fr{y^{1+c}}{T^{k+1}} \sum_{m \cg 0 \mm{q_0}} \fr{|f(m)|}{m^{1+c}} + \fr{1}{T^k} \sum_{m \cg 0 \mm{q_0}} |f(m)| \rt) .
 \end{split}
\end{equation}
Now, we shift the line of integration from $c$ to $-\gm$ ($0 < \gm < 1$) and collect the residue at $s = 0$. This makes the first term of \eqref{tmpI2} equal to:
\begin{equation}
\begin{split}
\label{tmpI3}
&\quad k! \, y \, \res_{s=0} \, \mc{F}_0(1 + s) \fr{y^s}{s^{k + 1}}  \biggl(\int_{\fr{1}{y}}^1 t^s e(\dl t) \, dt \biggr) \, + O\lf(\fr{y^{1 - \gm}}{1-\gm} \biggl( \int_{-T}^{T} \fr{du}{\bigl( \gm^2 + u^2 \bigr)^{\fr{k+1}{2}}}  \biggr) \sum_{m \cg 0 \mm{q_0}} \fr{|f(m)|}{m^{1 - \gm}} \rt) \\
&\qquad + O\lf( \fr{y^{1+c}}{(1-\gm) T^{k+1}} \sum_{m \cg 0 \mm{q_0}} \fr{|f(m)|}{m^{1-\gm}} \rt) .
\end{split}
\end{equation}
Now, to calculate the residue at $s=0$ in \eqref{tmpI3}, we express
\begin{equation*}
\int_{1/y}^1 t^s e(\dl t) \, dt = \sum_{l=0}^{\infty} \fr{s^l}{l!} \int_{1/y}^1 (\log t)^l e(\dl t) \, dt = \sum_{l=0}^{\infty} \fr{s^l}{l!} I_l(\dl, y) .
\end{equation*}
Therefore, the residue in \eqref{tmpI3} becomes:
\begin{equation}
\begin{split}
\label{tmpI4}
&\quad k! \, y \sum_{l=0}^k \fr{I_l(\dl, y) }{l!} \cdot \mathop{\mathrm{Res}}_{s=0} \, \mc{F}_0(1 + s) \fr{y^s}{s^{k-l+1}}
= y \sum_{l=0}^k \binom{k}{l} I_l(\dl, y) \sum_{m \cg 0 \mm{q_0}} \fr{f(m)}{m} \biggl( \log \fr{y}{m} \biggr)^{k-l}.
\end{split}
\end{equation}
Now, also note
\begin{equation*}
\begin{split}
 (-1)^l I_l(\dl, y) &= l \int_{1/y}^1 \lf( \int_{1}^{1/t} \fr{(\log u)^{l-1}}{u} \, du \rt) e(\dl t) \, dt = l \int_1^y \fr{(\log u)^{l-1}}{u} \biggl( \int_{1/y}^{1/u} e(\dl t) \, dt \biggr) \, du\\
&= O^*\lf(  \fr{(\log y)^l}{\dl_0}\rt),
\end{split}
\end{equation*}
since \footnote{One way to get rid of $(\log y)^l$, is by using $\min\bigl( \fr{1}{u}, \fr{1}{\pi|\dl|} \bigr) \leq 1/u$, but it would be at the cost of $\dl_0$.}
$| \int_{1/y}^{1/u} e(\dl t) \, dt| \leq \min\lf( \fr{1}{u}, \fr{1}{ \pi |\dl|} \rt) \leq 1/\dl_0$,  

Hence, \eqref{tmpI4} is at most
\begin{equation*}
\label{tmpIf}
\leq \fr{y}{\dl_0} \sum_{l=0}^k \binom{k}{l} (\log y)^{l} \lf| \sum_{m \cg 0 \mm{q_0}} \fr{f(m)}{m} \lf( \log \fr{y}{m} \rt)^{k-l} \rt| .
\end{equation*}
%
%
%
We now choose 
$$\gm = 1 - 1/(\log y), \quad c = 1/(\log y) \quad \text{and} \quad T = \begin{cases} y, & k =0, \\ \infty, & k \geq 1, \end{cases} $$ 
to find that the total $O$-term (combining all the error terms in \eqref{tmpI1}, \eqref{tmpI2} and \eqref{tmpI3}) is at most $\ll (\log y)^{\rho_k} \sum_{m \cg 0 \mm{q_0}} |f(m)|$, and the $O$-constant depends only on $k$. 
\end{proof}

\smallskip

The following (standard) results hold for Selberg sieve weights, under a divisibility condition.
\begin{lem}\text{ }
\label{lbsum}
\begin{itemize}
\item[(a)] We have
$$
B_r = \sum\lm_{d \cg 0 \mm{r}} \fr{\lb(d)}{d} = \begin{dcases} \fr{\mu(r)}{\vp(r) \, G_q(R)}, & r \leq R \text{ and } (r,q)=1,\\ 0, & \text{otherwise.} \end{dcases}
$$
\item[(b)] For $r \leq R$, $(r,q) = 1$, we have
$$
\sum_{d \cg 0 \mm{r}} \fr{\lb(d)}{d} \, \log d = B_r \Biggl( \log r - \sum\lm_{\st{p \leq R/r \\ (p, qr) = 1}} \fr{\log p}{p-1} \Biggr) = O^*\Biggl( (1 + \eps) \,\fr{\mu^2(r)}{\vp(r)} \fr{\log R}{G_q(R)} \Biggr).   
$$
\item[(c)] For $g \leq R$, $(g,q) = 1$:
$$
\sum_{d \cg 0 \mm{g}} |\lb(d)| = O\lf( \fr{\mu^2(g)}{\vp(g)} \fr{R}{G_q(R)} \rt).
$$
\end{itemize}
\end{lem}

\begin{proof}
The identity in (a) is a standard fact about the Selberg sieve. For $r \leq R$ and $(r,q) = 1$: 
$$
\sum\lm_{\st{d \cg 0 \mm{r}}} \fr{\lb(d)}{d} = \fr{1}{G_q(R)} \sum\lm_{\st{d \cg 0 \mm{r} \\ (d, q) = 1}} \mu(d) \sum\lm_{\st{\ell \leq R \\ \ell \cg 0 \md{d} \\ (\ell, q) = 1}} \fr{\mu^2(\ell)}{\vp(\ell)} = \fr{1}{G_q(R)} \sum\lm_{\st{\ell \leq R \\ \ell \cg 0 \mm{r} \\ (\ell, q) = 1 }} \fr{\mu^2(\ell)}{\vp(\ell)} \sum_{r \mid d \mid \ell} \mu(d) = \fr{\mu(r)}{\vp(r) \, G_q(R)} . 
$$
For (b), we write:
\begin{equation*}
\begin{split}
\sum_{d \cg 0 \mm{r}} \fr{\lb(d)}{d} \, \log d 
&= \sum_{d \cg 0 \mm{r}} \fr{\lb(d)}{d} \sum_{p \mid d} \log p
= \sum\lm_{\st{p \leq R \\ (p, q) = 1}} \log p \, B_{[p, r]} 
= B_r \Biggl( \log r  - \sum\lm_{\st{p \leq R/r \\ (p, q) = 1}} \fr{\log p}{p-1} \Biggr) \\
&= O^*\Biggl( (1 + \eps) \fr{\mu^2(r)}{\vp(r)} \fr{\log R}{G_q(R)}\Biggr),
\end{split}
\end{equation*}
where we use the fact that $\sum_{p \leq R/r} \fr{\log p}{p-1} = \log R/r + O(1)$ and that $|B_r| = \fr{\mu^2(r)}{\vp(r) \, G_q(R)}$. 

For (c), we have:
\begin{equation*}
\begin{split}
\sum_{d \cg 0 \mm{g}} |\lb(d)| &= \fr{1}{G_q(R)} \sum\lm_{\st{d \cg 0 \mm{g}\\ (d, q) = 1}} d \sum\lm_{\st{\ell \leq R \\ \ell \cg 0 \mm{d} \\ (\ell, q)=1}} \fr{\mu^2(\ell)}{\vp(\ell)}
= \fr{g}{G_q(R)} \sum\lm_{\st{\ell \leq R \\ \ell \cg 0 \mm{g} \\ (\ell, q)=1}} \fr{\mu^2(\ell) \sg(\ell/g)}{\vp(\ell)} \ll \fr{\mu^2(g)}{\vp(g)}\fr{R}{G_q(R)} . 
\end{split}
\end{equation*}
\end{proof}
\smallskip
\begin{lem}
Define $g_0$ on the square-free positive integers by $g_0(v) = \prod\lm_{p \mid v} \lf( 1 + \fr{1}{\sr{p}} \rt)$. Then, for any $0 < \rho <1/2$, we have 
$$g_0(v) \leq C(\rho) v^{\rho} .$$ 
\end{lem}

\begin{proof} Clearly the given function $g_0$ is bounded by the divisor function $\tau$, for which such a bound holds. 
%
%
\end{proof}

\smallskip

We now revisit the work of Ramar\'{e} \cite{Ram1} on partial sums of the M\"{o}bius function. Let 
$$\check{m}_v(X) = \sum\lm_{\st{n \leq X \\ (n,v)=1}} \fr{\mu(n)}{n} \log \fr{X}{n} \quad \text{ and } \quad \check{\check{m}}_v(X) = \sum\lm_{\st{n \leq X \\ (n, v)=1}} \fr{\mu(n)}{n} \log^2 \fr{X}{n}, $$ 
and set $\check{m}(X): = \check{m}_1(X)$ as well as $\check{\check{m}}(X) := \check{\check{m}}_1(X)$. 

\smallskip

We can now extract the following bounds from Theorems 1.5, 1.8 and Corollaries 1.10, 1.11 of \cite{Ram1} (for all $X \geq 1$):  
\begin{align*}
\check{m}(X) = 1  +  O^*\lf(\fr{c_1}{\log X}\rt) \quad &\text{and} \quad |\check{m}(X)| \leq c_1',\\
 \check{\check{m}}(X) = 2 \log X - 2\gm  + \,  O^*\lf(\fr{c_2}{\log X}\rt) \quad &\text{and} \quad \lf|\check{\check{m}}(X) \rt| \leq c'_2 \log X,  
\end{align*}

where $$c_1 = 0.213, \ \ c'_1 = 1.00303,  \ \ c_2 = 0.2062 \ \, \text{and} \ \,c'_2 = 2 .$$

\medskip

This allows us to deduce the following:

\begin{lem}
\label{muvlog}
Let $0 < \rho < 1/8$. Then, for $X \geq X_0(\rho)$ sufficiently large and square-free $v$, with $v \leq X^{1/\rho}$, one has
\begin{itemize}
\item[(a)]
$$
\check{m}_v(X) = \fr{v}{\vp(v)} \lf( 1 + O^*\lf( \fr{1}{\log X} \rt) \rt),
$$
\item[(b)]
$$
\check{\check{m}}_v(X) = \fr{2 v}{\vp(v)} \lf( \log X -  \gm - \sum_{p \mid v}\fr{\log p}{p-1} + O^*\lf( \fr{0.5}{\log X} \rt) \rt).
$$
\item[(c)] Consequently, one has
$$
 \sum\lm_{\st{d \leq X \\ (d, v) = 1}} \fr{\mu(d)}{d} \log \fr{X}{d} \, \log d  = \fr{v}{\vp(v)} \lf( -\log X + 2\sum_{p \mid v} \fr{\log p}{p-1} + O(1) \rt) .
$$
\end{itemize}
\end{lem}

\begin{proof}
Let us prove (a). It is easy to see that:
\begin{equation}
\begin{split}
\label{t2mp1}
\check{m}_v(X)  &=  \sum\lm_{\st{ \ell \mid v^{\infty}\\ \ell \leq X}} \fr{\check{m}(X/\ell)}{\ell} 
= \sum\lm_{\st{\ell \mid v^{\infty} \\ \ell \leq X^{3/4}}} \fr{1}{\ell} \lf( 1 + 
O^*\lf( \fr{c_1}{\log X/\ell} \rt) \rt) \, 
+ \, O^*\Biggl(c'_1\sum\lm_{\st{\ell \mid v^{\infty} \\ \ell > X^{3/4}}} \fr{1}{\ell} \Biggr)\\
&= \fr{v}{\vp(v)} \lf( 1 + O^*\lf( \fr{4c_1}{\log X} \rt) \rt)  
+  O^*\Biggr( \lf( 1 + c_1' + \fr{4c_1}{\log X} \rt) 
\sum\lm_{\st{\ell \mid v^{\infty} \\ \ell > X^{3/4}}} \fr{1}{l} \Biggr)
\end{split}
\end{equation}
Note that,
\begin{equation*}
\begin{split}
\sum\lm_{\st{\ell \mid v^{\infty} \\ \ell > X^{3/4}}} \fr{1}{\ell} &\leq \fr{1}{X^{3/8}} \sum_{\ell \mid v^{\infty}} \fr{1}{\sr{\ell}}  = \fr{1}{X^{3/8}}\prod_{p \mid v} \lf( 1 - \fr{1}{\sr{p}} \rt)^{-1}
= \fr{1}{X^{3/8}} \fr{v}{\vp(v)} \prod_{p \mid v} \lf( 1 + \fr{1}{\sr{p}} \rt)\\
&\leq \fr{C(\rho/4) \, v^{\rho/4}}{X^{3/8}} \fr{v}{\vp(v)} \leq \fr{C(\rho/4) }{X^{1/8}} \fr{v}{\vp(v)} \leq \fr{v}{\vp(v)} \fr{\eps}{\log X},
\end{split}
\end{equation*}
since $v \leq X^{1/\rho}$. Therefore, the error term in the last line of \eqref{t2mp1} is at most $\fr{\eps (1 + c'_1 + \eps)}{\log X} \fr{v}{\vp(v)}$, 
for all $X \geq X_0(\rho)$ sufficiently large, which proves (a).

\vskip 0.01in

Now, we prove (b). Again, we write:
\begin{equation*}
\begin{split}
\check{\check{m}}_v(X) 
&= \sum\lm_{\st{\ell \mid v^{\infty} \\ \ell \leq X}} \fr{\check{\check{m}}(X/\ell)}{\ell}
= 2 \sum\lm_{\st{\ell \mid v^{\infty} \\ \ell \leq X^{3/4}}} \fr{1}{\ell} \lf( \log \fr{X}{\ell} - \gm + O^*\lf( \fr{c_2/2}{\log X^{1/4}} \rt) \rt) \, +  \, O^*\Biggl(c'_2 \log X  \sum\lm_{\st{\ell \mid v^{\infty} \\\ell > X^{3/4}}}\fr{1}{\ell} \Biggr) \\
&= 2 \sum\lm_{\st{\ell \mid v^{\infty}}} \fr{1}{\ell}\lf( \log \fr{X}{\ell} - \gm  + O^*\lf( \fr{2c_2}{\log X} \rt) \rt) \,  + \, O^*\Biggl( (1 + c'_2) \log X  \sum\lm_{\st{\ell \mid v^{\infty} \\\ell > X^{3/4}}}\fr{1}{\ell}\Biggr)
\end{split}
\end{equation*}
Now, just like (a), the error term is $O^*\lf(\fr{v}{\vp(v)}\fr{\eps}{\log X}\rt)$. And the main term is:
\begin{equation*}
\fr{2v}{\vp(v)}\lf(\log X - \gm  - \sum_{p \mid v} \fr{\log p}{p-1} + O^*\lf( \fr{2c_2}{\log X} \rt) \rt),
\end{equation*}
since $\sum_{\ell \mid v^{\infty}} \fr{\log \ell}{\ell} = \fr{v}{\vp(v)} \sum_{p \mid v} \fr{\log p}{p-1}$. Combining these two bounds, we prove (b).

And finally, (c) follows from (a) and (b), since 
$\log X/d \, \log d = \log X/d \, \log X - \log^2 X/d$. This completes the proof.
\end{proof}
\smallskip

%
\begin{lem}
\label{thtsum}
Let $0 < \rho < 1/8$ and $v$ be a square-free positive integer with $U \geq v^{1 + \rho}$. Then:
\begin{itemize}
\item[(a)]
$$
\sum\lm_{d \cg 0 \mm{v}} \fr{\tht'(d)}{d} = O^*\lf( \fr{2}{\vp(v) \, \log \fr{U_1}{U} \, \log \fr{U}{v}} \rt),
$$
\item[(b)]
$$
\sum\lm_{d \cg 0 \mm{v}} \fr{\tht'(d) \, \log d}{d}  = \fr{\mu(v)}{\vp(v)} \lf( -1 + O\biggl( \fr{\rho^{-1}}{\log \fr{U_1}{U}}\biggr) \rt),
$$
\item[(c)]
$$
\sum_{d \cg 0 \mm{v}} |\tht'(d)| \leq \fr{U_1}{v \, \log \fr{U_1}{U}}. 
$$
\end{itemize}
\end{lem}
\begin{proof}
Using Lemma \ref{muvlog} (a), we will obtain:
\begin{equation*}
\begin{split}
\label{ttmp8}
\sum\lm_{d \cg 0 \mm{v}} \fr{\tht'(d)}{d} &= \fr{\mu(v) }{v \, \log \fr{U_1}{U}} \lf( \sum\lm_{\st{ d \leq U_1/v \\ (d, v) = 1 }} \fr{\mu(d)}{d} \log \fr{U_1/v}{d} - \sum\lm_{\st{ d \leq U/v \\ (d, v) = 1}} \fr{\mu(d)}{d} \log \fr{U/v}{d}  \rt)\\ 
&=  \fr{\mu(v) }{v \, \log \fr{U_1}{U} } \fr{v}{\vp(v)} \cdot O^*\lf( \fr{2}{\log U/v} \rt) 
= O^*\lf( \fr{2}{\vp(v) \, \log \fr{U_1}{U} \, \log \fr{U}{v} } \rt),
\end{split}
\end{equation*}
provided $v \leq (U/v)^{1/\rho}$, which is equivalent to $U \geq v^{1 + \rho}$. 

Next, we have (using Lemma \ref{muvlog} (a) and (c)): 
\begin{equation*}
\begin{split}
\sum\lm_{\st{d \cg 0 \mm{v}}} \fr{\tht'(d) \, \log d}{d}
&= \fr{\mu(v)}{v \, \log \fr{U_1}{U}} \lf( \sum\lm_{\st{d \leq U_1/v \\ (d, v) = 1}} \fr{\mu(d)}{d} \log \fr{U_1/v}{d} \, \log dv - \sum\lm_{\st{d \leq U/v \\ (d, v) = 1}} \fr{\mu(d)}{d} \log \fr{U/v}{d} \, \log dv \rt) \\
&= \fr{\mu(v)}{\vp(v) \, \log \fr{U_1}{U}}
\lf( O^*\biggl( \fr{2 \, \log v }{\log \fr{U}{v}} \biggr)  - \log \fr{U_1}{U} + O(1)\rt) 
= \fr{\mu(v)}{\vp(v)} \lf( -1 + O\biggl( \fr{\rho^{-1}}{\log \fr{U_1}{U}} \biggr) \rt) .
\end{split}
\end{equation*}

For (c), we note that
\begin{equation*}
\sum_{d \cg 0 \mm{v}} |\tht'(d)| \leq \fr{1}{\log \fr{U_1}{U}} \sum\lm_{\st{d \leq U_1/v \\ (d, v) = 1}} \log \fr{U_1/v}{d} 
\leq \fr{U_1}{v \, \log \fr{U_1}{U}},
\end{equation*}
where we use the inequality $\sum_{n \leq z} \log z/n \leq  z$, with $z = U_1/v$. 
\end{proof}

\smallskip

\begin{lem}
\label{qhm}
Suppose $q_0 \mid q$, $U < y$ and $0 < \rho < 1/8$. Then, we have the following:
\begin{itemize}
\item[(a)] If $U > (q_0 R)^{1 + \rho}$, we have
$$
\lf| \sum_{m \cg 0 \mm{q_0}} \fr{h(m)}{m} \rt| \leq \fr{2}{\vp(q_0) \, \log \fr{U_1}{U} \, \log \fr{U}{q_0 R}}.
$$
\item[(b)] If $U > (q_0 R)^{1 + \rho}$, we have
\begin{equation*}
\sum_{m \cg 0 \mm{q_0}} \fr{h(m)}{m} \log \fr{y}{m} = 
\fr{\mu(q_0)}{\vp(q_0)} \lf( 1 + O\biggl(\fr{\rho^{-1} \, \log y}{\log \fr{U_1}{U} \, \log q_0R } \biggr) \rt) .
\end{equation*}

\item[(c)] 
$$
\sum_{m \cg 0 \mm{q_0}} |h(m)|  = O\lf( \fr{q}{\vp(q)} \fr{U_1 R}{q_0 \, \log R \, \log \fr{U_1}{U}} \rt).
$$
%
%
\end{itemize}
\end{lem}
\begin{proof}
For (a), we write
\begin{equation*}
\begin{split}
\sum_{m \cg 0 \mm{q_0}} \fr{h(m)}{m} &= \sum\lm_{\st{d_1, d_2 \\ q_0 \mid [d_1, d_2]}} \fr{\lb(d_1) \tht'(d_2) }{[d_1, d_2]}
 = \sum\lm_{\st{r \leq R \\ (r, q) = 1}} \vp(r) \lf( \sum\lm_{\st{d_1 \cg 0 \mm{r}}} \fr{\lb(d_1)}{d_1} \rt) \lf( \sum\lm_{\st{d_2 \cg 0 \mm{q_0 r}}} \fr{\tht'(d_2)}{d_2} \rt),
\end{split}
\end{equation*}
since the condition $q_0 \mid [d_1, d_2]$ implies $q_0 \mid d_2$ (because $(d_1, q) = 1 = (d_1, q_0)$, as $q_0 \mid q$). 

%
Therefore, by Lemma \ref{lbsum} (a) and Lemma \ref{thtsum} (a) (with $v = q_0 r$), we have
\begin{equation*}
\label{ttmp1}
\begin{split}
\lf| \sum_{m \cg 0 \mm{q_0}} \fr{h(m)}{m} \rt| 
&\leq \fr{2}{\vp(q_0) \, G_q(R) \, \log \fr{U_1}{U} \, \log \fr{U}{q_0R}} \sum\lm_{\st{r \leq R \\ (r, q) = 1}} \fr{\mu^2(r)}{\vp(r)} 
= \fr{2}{\vp(q_0) \, \log \fr{U_1}{U} \, \log \fr{U}{q_0R}} .
\end{split}
\end{equation*}
%
%
%
%

Now, we prove (b). We express the given sum as 
\begin{equation}
\begin{split}
\label{S1init}
S &= \sum\lm_{\st{d_1, d_2 \\ q_0 \mid [d_1, d_2]}} \fr{\lb(d_1) \tht'(d_2)}{[d_1, d_2]} \log \fr{y}{[d_1, d_2]}\\
&= \sum\lm_{\st{r \leq R \\ (r, q) = 1}} \vp(r) \sum\lm_{\st{d_1 \cg 0 \mm{r} \\ d_2 \cg 0 \mm{q_0 r}}} 
\fr{\lb(d_1)}{d_1} \fr{\tht'(d_2)}{d_2} \bigl( \log y - \log d_1 - \log d_2 + \log (d_1, d_2) \bigr)\\
&= S_0 \, -  \, S_1 \, - \, S_2 \, + \, S_3, 
\end{split}
\end{equation}

Now, from (a), we have
\begin{equation}
\label{S0b}
S_0 = \log y \sum_{m \cg 0 \mm{q_0}} \fr{h(m)}{m} = O^*\lf( \fr{2 \, \log y}{\vp(q_0) \, \log \fr{U_1}{U} \, \log \fr{U}{q_0 R}} \rt) .
\end{equation}
Next, one has (using Lemma \ref{lbsum} (b) and Lemma \ref{thtsum} (a))
\begin{equation}
\label{S1b}
\begin{split}
S_1 
&= \sum\lm_{\st{r \leq R \\ (r, q) = 1}} \vp(r) \lf( \sum_{d_1 \cg 0 \mm{r}} \fr{\lb(d_1)}{d_1} \log d_1 \rt) \lf( \sum_{d_2 \cg 0 \mm{q_0r}} \fr{\tht'(d_2)}{d_2}\rt)\\
&= O^*\Biggl( \fr{(1 + \eps) \, \log R}{G_q(R)} \sum\lm_{\st{r \leq R \\ (r, q) = 1}} \mu^2(r) \cdot \fr{2}{\vp(q_0) \vp(r) \, \log \fr{U_1}{U} \, \log \fr{U}{q_0 R}} \Biggr) 
= O^*\Biggl( \fr{(2 + \eps) \, \log R}{\vp(q_0) \, \log \fr{U_1}{U} \, \log \fr{U}{q_0 R}} \Biggr)
\end{split}
\end{equation}
Next, we have (using Lemma \ref{lbsum} (c) and Lemma \ref{thtsum} (b))
\begin{equation}
\label{S2b}
S_2 = \fr{1}{G_q(R)}\sum\lm_{\st{r \leq R \\ (r, q) = 1}} \mu(r) \lf( \sum_{d_2 \cg 0 \mm{q_0r}} \fr{\tht'(d_2)}{d_2}  \log d_2 \rt)
= \fr{\mu(q_0)}{\vp(q_0)}\lf( -1 + O\lf( \fr{\rho^{-1}}{\log \fr{U_1}{U}} \rt) \rt).
\end{equation}

It now remains to estimate $S_3$. We have
\begin{equation}
\label{S3b}
\begin{split}
S_3 &= \sum\lm_{\st{[p, r] \leq R \\ ([p, r], q) = 1}} \vp(r) \log p \cdot B_{[p, r]} \cdot O^*\lf( \fr{2}{\vp(q_0) \vp([p, r]) \, \log \fr{U_1}{U} \, \log \fr{U}{q_0 [p, r]} } \rt) \\
&= O^*\Biggl( \fr{2}{\vp(q_0) \, G_q(R) \, \log \fr{U_1}{U} \, \log \fr{U}{q_0 R}} \sum\lm_{\st{[p, r] \leq R \\ (p,q) = (r, q) = 1}} \fr{\mu^2([p, r]) \, \vp(r)}{\vp([p, r])^2} \, \log p \Biggr)\\
&= O^*\Biggl( \fr{2}{\vp(q_0) \, G_q(R) \, \log \fr{U_1}{U} \, \log \fr{U}{q_0 R}} \sum\lm_{\st{r \leq R \\ (r, q) = 1}} \fr{\mu^2(r)}{\vp(r)} \biggl( \log r + \sum\lm_{\st{p \leq R/r \\ (p, q) = 1}} \fr{\log p}{(p-1)^2} \biggr) \Biggr)\\
&= O^*\lf( \fr{(2 + \eps) \, \log R}{\vp(q_0) \, \log \fr{U_1}{U} \, \log \fr{U}{q_0 R}} \rt),
\end{split}
\end{equation}
%
%
provided we have $U \geq (q_0R)^{1 + \rho}$ (since we use Lemma \ref{thtsum} (a) with $v = q_0 [p, r] \leq q_0 R$).  

From \eqref{S1init}, \eqref{S0b}, \eqref{S1b}, \eqref{S2b}, \eqref{S3b}, and noting that $\log q_0R \leq \rho^{-1} \log \fr{U}{q_0 R}$, we prove (b). 

To prove (c), write
\begin{equation*}
\begin{split}
\sum_{m \cg 0 \mm{q_0}} |h(m)| = \Biggl(  \sum_{d_1} |\lb(d_1)|  \Biggr)  \Biggl( \sum_{d_2 \cg 0 \mm{q_0}} |\tht'(d_2)| \Biggr)
\ll \fr{q}{\vp(q)} \fr{U_1 R}{q_0  \,  \log R \, \log \fr{U_1}{U}},
\end{split}
\end{equation*}
by an application of lemmas \ref{lbsum} (c) and \ref{thtsum} (c), and noting that $q_0 \mid [d_1, d_2]$ implies $q_0 \mid d_2$. 
\end{proof}

The next lemma is similar to \cite[Lemma 11.11]{HH}, with our bounds being slightly different. 

\begin{lem}
\label{LbV}
We have:
\begin{itemize}
\item[(a)]
$$
\sum_{\ell \leq V} \fr{\Lambda(\ell)}{\ell} (\ell, q) \leq \log Vq + O(1) \quad \text{ and } \quad \sum_{\ell \leq V} \Lambda(\ell) (\ell, q) = O\bigl(V + q\log Vq  \bigr)
$$
\item[(b)]
$$
\sum_{\ell \leq V} \fr{\mu^2(\ell)}{\ell} (\ell, q) \leq \tau(q) \log eV \quad \text{ and } \quad \sum_{\ell \leq V} \mu^2(\ell) (\ell, q) \leq \tau(q) \, V. 
$$
\end{itemize}
\end{lem}
\begin{proof}
We begin with (a). For the first bound, we have:
\begin{equation}
\begin{split}
\label{ttmp5}
\sum_{\ell \leq V} \fr{\Lb(\ell)}{\ell} (\ell, q) &= \sum\lm_{\st{\ell \leq V \\ (\ell, q) = 1}} \fr{\Lb(\ell)}{\ell} \, + \, \sum\lm_{\st{\ell \leq V \\ \ell \mid q^{\infty}}} \fr{\Lb(\ell)}{\ell} (\ell, q). 
\end{split}
\end{equation}
By Mertens estimate, the first term in \eqref{ttmp5} is at most $\log V  - \sum_{p \mid q} \fr{\log p}{p} +  O(1)$. The second term is:
\begin{equation*}
\begin{split}
\sum\lm_{\st{\ell \leq V \\ \ell \mid q^{\infty}}} \fr{\Lb(\ell)}{\ell} (\ell, q) &\leq \sum_{p \mid q} \log p \sum_{j = 1}^{\infty} \fr{p^{\min(j, v_p(q))}}{p^j} = \sum_{p \mid q} \log p \lf( \sum_{j = 1}^{v_p(q)} \fr{p^{j}}{p^j} \, + \, \fr{p^{v_p(q)}}{p^{v_p(q)}}  \sum_{j=1}^{\infty} \fr{1}{p^j}  \rt) \\
&\leq \sum_{p \mid q} \lf( v_p(q) + \fr{1}{p-1} \rt) \log p = \log q + \sum_{p \mid q} \fr{\log p}{p-1} .
\end{split}
\end{equation*}
Combining the two estimates, we get a bound of $\log Vq + O(1)$, as desired. 

\smallskip
Let us prove the second bound. We have
\begin{equation}
\label{ttmp6}
\sum_{\ell \leq V} \Lb(\ell) (\ell, q) = \sum\lm_{\st{\ell \leq V \\ (\ell, q) = 1}} \Lb(\ell) \, + \, \sum\lm_{\st{\ell \leq V \\ \ell \mid q^{\infty}}} \Lb(\ell) (\ell, q). 
\end{equation}
The first term in \eqref{ttmp6} is clearly $O(V)$. The second term is at most:
\begin{equation*}
\begin{split}
&\quad \sum_{p \mid q} \log p \sum_{j = 1}^{\floor{\log_p V}} p^{\min(j, v_p(q))} \leq \sum_{p \mid q} \log p \lf( \sum_{j = 1}^{v_p(q)} p^j \, + \, \sum_{j = v_p(q)+1}^{\floor{\log_p V}} p^{v_p(q)} \rt) \\
&\leq \sum_{p \mid q}p^{v_p(q)} v_p(q) \log p \, + \, \log V \sum_{p \mid q} p^{v_p(q)} \leq q \log q \, + \, q \log V \leq q \log Vq, 
\end{split}
\end{equation*}
where we use $\sum_{p \mid q} p^{v_p(q)} \leq q$. 

\smallskip

Now, we prove (b). We write
\begin{equation*}
\begin{split}
\sum_{\ell \leq V} \fr{\mu^2(\ell)}{\ell}(\ell, q) &\leq \sum_{g \mid q} \sum_{\ell \leq V/g} 1/\ell \leq \tau(q) \log eV.
\end{split}
\end{equation*}
Similarly,
\begin{equation*}
\begin{split}
\sum_{\ell \leq V} \mu^2(\ell) (\ell, q) &\leq V \sum_{g \mid q} 1 \leq \tau(q) V.
\end{split}
\end{equation*}
This completes the proof. 
\end{proof}
\smallskip

Next, we state a version of Abel's inequality:
\begin{lem}[Abel's inequality]
\label{AI}
Let $\{a_n\}_{n \geq 1}$ be any sequence of complex numbers and let $\phi : [a,b] \to \mb{R}$ be non-negative and non-decreasing. Then
$$
\Biggl| \sum_{a < n \leq b}  a_n \phi(n)   \Biggr| \leq \phi(b) \cdot \Biggl(\max\lm_{c \in (a, b]} \biggl|\sum_{c < n \leq b} a_n \biggr| \Biggr)
$$
\end{lem}
\begin{proof}
This is standard. 
Suppose that $A = \max_{c \in (a, b]}|\sum_{c < n \leq b} a_n|$ and assume that $(a,b] \cap \mb{N} = \{K+1, \ldots, K + L\}$. We can now write:
\begin{equation*}
\sum_{a < n \leq b} a_n  \phi(n) = \phi(K+1) \biggl(\sum_{j=K+1}^{K + L} a_j \biggr)  + \sum_{l=K + 1}^{K + L-1} \bigl( \phi(l+1) - \phi(l) \bigr) \biggl(\sum_{j = l + 1}^{K + L} a_j \biggr). 
\end{equation*}
Now, since $\phi$ is non-decreasing and all partial sums $\sum_{l+1}^{K + L} a_j$ are bounded by $A$, the above (in absolute value) is at most $A \biggl(\phi(K+1) + \sum\lm_{l=K+1}^{K+L-1} \bigl(\phi(l+1) - \phi(l)\bigr) \biggr) = A \phi(K+L) \leq A \phi(b)$. 
\end{proof}
\smallskip
To handle the trigonometric quantities, we use the Abel's inequality, to get
\begin{equation} 
\label{AIapp} 
\lf| \sum_{n \leq y/m} (\log n)^k \, e(mn\al) \rt| \leq \bigl( \log y/m \bigr)^k \cdot \min \lf (\fr{y}{m}, \fr{1}{|\sin \pi m \al|}\rt) .
\end{equation}

We need bounds for the sum over the RHS above over $m$'s in certain ranges. Various results are found in \cite{Vin54}, \cite{DR}, \cite{Tao}, \cite{HH}. We state the following result from Daboussi and Rivat \cite{DR} (sharper results are present in \cite{HH}, but they appear with a smoothing and have a $1/\sin^2$, whereas we need a $1/\sin$):
\begin{lem}[\text{\cite[Lemma 1]{DR}}]
\label{SY}
Let $\al$ be as in \eqref{AP0} and let $0 <z_1 < z_2$ with $z_2 - z_1 \leq q_0$. Then
$$
\sum_{z_1 < m \leq z_2} \min\lf( A, \fr{1}{|\sin \pi m \al|} \rt) \leq 2A + \fr{2 q_0}{\pi}\log 4q_0.
$$ 
\end{lem}

Next, we also need a bound where the sum over $m$-sum runs over numbers not divisible by $q_0$. The following (slightly modified) result from the work of Helfgott comes to our rescue:
\begin{lem}[\text{\cite[Lemma 11.1]{HH}}]
\label{SqmY}
Let $\al$ be as in \eqref{AP0} and let $0 < z_1 < z_2$ with $z_2 - z_1 \leq q_0$. If $z_2 \leq \fr{y}{2|\dl|q_0}$, then\footnote{The condition given in \cite[Lemma 11.1]{HH} is $z_2 \leq Q_0/2$, but all that is needed is $z_2 \leq \fr{y}{2|\dl|q_0}$.} 
$$ \sum\lm_{\st{z_1 < m \leq z_2 \\ q_0 \nmid m}} \fr{1}{|\sin \pi m \al|} \leq  q_0 \log eq_0 .$$
\end{lem}

The next lemma is a direct consequence the Euler summation formula:
\begin{lem}
\label{Yrho}
Let $X > \rho > 0$ and $r \geq 0$. Then
\begin{equation*}
\begin{split}
\sum_{0 \leq m \leq X -  \rho} \lf( \log \fr{X}{m + \rho} \rt)^r &\leq r! X \, + \, \lf( \log X/\rho \rt)^r, \\
\sum_{0 \leq m \leq X - \rho} \fr{1}{m + \rho} \lf( \log \fr{X}{m + \rho} \rt)^r &\leq \fr{(\log X/\rho)^{r+1}}{r+1} \, + \, \rho^{-1} \lf( \log X/\rho \rt)^r. 
\end{split}
\end{equation*}
\end{lem}
\begin{proof}
This is a special case of \cite[Lemma 2.20]{PS}. 
\end{proof}

\smallskip

This leads to the following bounds:
\begin{lem}
\label{mYX}
Let $\al$ be as in \eqref{AP0}, $3 < Y \leq  X$ and $r \geq 0$. Then:  
\begin{itemize}
\item[(a)]
%
\begin{equation*}
\begin{split}
\sum\lm_{\st{m \leq Y}} (\log Y/m)^r \min\lf( \fr{y}{m}, \fr{1}{|\sin \pi m \al|} \rt)  &\leq  \log 4q_0\lf( \fr{2 r!}{\pi} Y \, + \, 2q_0 (\log Y)^r \rt) \\
&\quad \, + \, \fr{2y}{q_0}\bigl( \log^+ 2Y/q_0 \bigr)^r \lf( \fr{\log^+ 2Y/q_0}{r+1}  + 2 \rt).
\end{split}
\end{equation*}

\item[(b)]
If $Y \leq \fr{y}{2 |\dl| q_0}$, we have
%
$$
\sum\lm_{\st{m \leq Y \\ q_0 \nmid m}} \fr{(\log Y/m)^r}{|\sin \pi m \al|}  \leq \log eq_0 \, \bigl( r! \, Y + q_0 (\log Y)^r \bigr).
$$
\end{itemize}
\end{lem}
\begin{proof}
For (a), we follow the proof of \cite[Lemma 3]{DR}. Split the given sum as 
$$
\sum_{m \leq Y} = \sum_{m \leq q_0/2} + \sum_{j=0}^{\floor{\fr{Y}{q_0} - \fr{1}{2}}} \sum_{q_0(j + 1/2) < m \leq q_0(j + 3/2)} = T_1 \, + \, T_2
$$
In the first sum (where $m \leq q_0/2$), we can use Lemma \ref{SqmY}, since $q_0 \nmid m$ in this range. This would imply $|T_1| \leq q_0 \log eq_0 \, (\log Y)^r$. 

Now, for $T_2$ (which survives only when $Y \geq q_0/2$), we use Lemma \ref{SY}, to obtain:
\begin{equation*}
\begin{split}
|T_2| & \leq \sum\lm_{j=0}^{\floor{\fr{Y}{q_0} - \fr{1}{2}}} \biggl(\log^+ \fr{Y/q_0}{j + 1/2} \biggr)^r
\lf( \fr{2y}{q_0(j + 1/2)} + \fr{2q_0}{\pi} \log 4q_0 \rt)\\
&\leq \fr{2y}{q_0} \lf( \fr{(\log^+ 2Y/q_0)^{r+1}}{r+1}  + 2 (\log^+ 2Y/q_0)^r \rt) \, + \, \fr{2q_0}{\pi} \log 4q_0 
\lf( \fr{r! Y}{q_0}  + (\log^+ 2Y/q_0)^r \rt)\\
&\leq \fr{2 \log 4q_0}{\pi} \bigl( r! Y \, + \, q_0 (\log Y)^r \bigr) \, + \, \fr{2y}{q_0}\bigl( \log^+ 2Y/q_0 \bigr)^r \lf( \fr{\log^+ 2Y/q_0}{r+1}  + 2 \rt) ,
\end{split}
\end{equation*}
where we used Lemma \ref{Yrho} with $X = Y/q_0$ and $\rho = 1/2$ in the second line. 

Combining this with the bound for $|T_1|$, we prove (a). 

For (b), write the given sum as:
\begin{equation*}
\begin{split}
&\quad \sum\lm_{\st{m \leq q_0 \\ q_0 \nmid m}} \fr{(\log Y/m)^r}{|\sin \pi m \al|} \, + \, \sum\lm_{\st{q_0 < m \leq Y \\ q_0 \nmid m}} \fr{(\log Y/m)^r}{|\sin \pi m \al|}
\leq q_0 \log eq_0  \lf( (\log Y)^r  +  \sum_{1 \leq j \leq Y/q_0} \lf( \log \fr{Y}{jq_0} \rt)^r \rt) \\ 
&\leq \log eq_0 \, \bigl(r! \, Y  + q_0 (\log Y)^r \bigr),
\end{split}
\end{equation*}
where we use $\sum_{j \leq Y/q_0} (\log Y/jq_0)^{r} \leq \fr{Y}{q_0} r!$. This completes the proof. 
\end{proof}
\smallskip
We now refer to the previously discussed alternate approximation Lemma\footnote{A similar result is present in the book of Vaughan \cite[Pg 25, Ex 2]{RCVb}.} from \cite{HH}. 
\begin{lem}[\text{\cite[Lemma 2.2]{HH}}]
\label{alt}
Let $\al$ be as given in \eqref{AP0} with $\dl \neq 0$. Then there exists an approximation 
$$\al = a_0'/q_0' + \dl'/y, \quad \text{with } \ \, (a_0', q_0') = 1,\quad \fr{y}{2|\dl|q_0} \leq q_0' \leq \fr{y}{|\dl|q_0} = Q_{\dl}' \quad \text{ and } \quad |\dl'|/y \leq 1/q_0' Q_{\dl}',
$$
\end{lem}
\smallskip
Now, using the above lemmas, we prove the following results on the type-I contributions coming from those $m$'s not divisible by $q_0$. We have already handled the $m$'s divisible by $q_0$ in Proposition \ref{hqm}. Again, the next Lemma may be compared to \cite[Lemma 11.10]{HH}.
\begin{lem}
\label{sumdD}
Let $\al$ be as in \eqref{AP0} with $q_0 \mid q$ and let $D \gg R \dl_0 q_0$ as well as $R \gg x^{\eta/4}$. Then
$$
\sum\lm_{\st{d \leq D \\ q_0 \nmid d }} \log \fr{D}{d} \sum_{n \leq y/d} (\log n)^k \, e(n d \al) 
\ll D \, (\log y)^k \cdot \begin{cases} \log eq_0, & \text{if } D \leq \fr{y}{2|\dl|q_0}, \\ (\log D)^2, & \text{if } D > \fr{y}{2 |\dl| q_0}. \end{cases}
$$
\end{lem}
\begin{proof}
Let us denote the given sum by $S$. Suppose first that $D \leq \fr{y}{2 |\dl| q_0}$. In this case, we may use Lemma \ref{mYX} (b) (with $r=1$ and $Y =D$), to get
$$
|S| \leq \log eq_0 \, (\log y)^k \lf( D + q_0 \log D \rt)  \ll D \, \log eq_0 \, (\log y)^k .
$$

So, assume that $D > \fr{y}{2 |\dl| q_0}$. We seek an alternate approximation (Lemma \ref{alt}) for $\al$. This gives us a new approximation $\al = a_0'/q_0' + \dl'/y$, with $\fr{y}{2 |\dl| q_0} \leq q_0' \leq \fr{y}{|\dl|q_0} < 2 D$ (and $\al$ satisfies \eqref{AP0} with this new approximation). We now apply Lemma \ref{mYX} (a) (with $Y  = D$, $r = 1$ and $\al = a_0'/q_0' + \dl'/y$) to get
\begin{equation*}
\begin{split}
|S| &\leq 2 \log 4q_0' \lf( D/\pi + q_0' \log D \rt) (\log y)^k \, +  \, \fr{2y}{q_0'} \log \fr{2D}{q_0'} \lf( 2 + \log \fr{2D}{q_0'} \rt) (\log y)^k  \\
&\ll  (\log y)^k (\log D)^2  \lf( D + |\dl|q_0 \rt) 
\ll D \, (\log D)^2 \, (\log y)^k .
\end{split}
\end{equation*}
\end{proof}
%
\begin{cor0}
\label{hqnm}
Let $\al$ be as in \eqref{AP0} with $q_0 \mid q$ and $q_0 \leq Q_0/R$. Suppose $U_1 \gg R^2 \dl_0 q_0$ and that $R \gg x^{\eta/4}$. Then, for all $k  \geq 0$, we have:
$$
\lf| \sum\lm_{\st{ m \\ q_0 \nmid m }} h(m)  \sum_{mn \leq y} (\log n)^k \, e(mn\al) \rt| \ll \fr{q}{\vp(q)} \fr{U_1 R \, (\log y)^k}{\log R \, \log \fr{U_1}{U}} \cdot 
\begin{cases} 
\log eq_0, & U_1R \leq \fr{y}{2 |\dl| q_0}, \\
\log^2 U_1, & U_1R > \fr{y}{2 |\dl| q_0}.
 \end{cases}
$$
\end{cor0}

\begin{proof}
First, we write $m = [gd_1, gd_2] = gd_1d_2$ (assuming $(d_1, d_2) = 1$) and we later ignore the condition $(d_1, d_2) = 1$. Note that $|\tht'(gd_2)| \leq \fr{\log U_1/(gd_2)}{\log U_1/U}$. Hence, the sum is at most
\begin{equation}
\begin{split}
\label{t3mp1}
&\fr{1}{\log \fr{U_1}{U}}
\sum\lm_{\st{g \leq R \\ (g, q) = 1}} \sum\lm_{\st{d_1 \leq R/g \\ (d_1, gq) = 1}} |\lb(gd_1)| \sum\lm_{\st{ d_2 \leq U_1/g \\ (d_2, g) = 1 \\ q_0 \nmid d_2}} \log \fr{U_1/g}{d_2} \lf| \sum\lm_{n \leq \fr{y}{gd_1d_2}} (\log n)^k \, e(n d_2 (gd_1 \al)) \rt| . 
\end{split}
\end{equation}
Now, since $(gd_1, q) = 1 = (gd_1, q_0)$, it follows that 
$$
gd_1\al = a_1/q_0 + O^*\lf(\fr{|\dl|}{y/(gd_1)} \rt) , \quad \text{with } \, (a_1, q_0) = 1, \ \ \fr{|\dl|}{y/(gd_1)} \leq \fr{1}{q_0(Q_0/R)} \ \, \text{and} \ \, q_0 \leq Q_0/R . 
$$ 
Thus, $gd_1\al$ satisfies \eqref{AP0} with $y/(gd_1)$ in place of $y$ and $Q_0/R$ in place of $Q_0$. 

Now, if $U_1R \leq \fr{y}{2 |\dl| q_0}$, it follows that $\fr{U_1}{g} \leq \fr{y/(gd_1)}{2 |\dl|q_0}$. We can then apply the first bound of Lemma \ref{sumdD} to the inner sum over $d_2$ in \eqref{t3mp1} (with $D = U_1/g \gg R\dl_0 q_0$, with $y/(gd_1)$ replacing $y$ and $gd_1\al$ replacing $\al$), so that \eqref{t3mp1} is at most
\begin{equation*}
\begin{split}
&\ll \fr{\log eq_0 \, (\log y)^k }{\log \fr{U_1}{U}}\sum\lm_{\st{g \leq R \\ (g, q) = 1}} \sum\lm_{\st{d_1 \leq R/g}} |\lb(gd_1)|
\cdot  \fr{U_1}{g} 
\ll \fr{q}{\vp(q)} \fr{U_1 R \, \log eq_0 \, (\log y)^k}{\log R \, \log \fr{U_1}{U}},
\end{split}
\end{equation*}
%
%
where we used $G_q(R) \geq \fr{\vp(q)}{q} \log R$. 

And, in case $U_1R > \fr{y}{2|\dl|q_0}$, both $\fr{U_1}{g} \leq \fr{y/(gd_1)}{2 |\dl| q_0}$ or $\fr{U_1}{g} > \fr{y/(gd_1)}{2 |\dl| q_0}$ could possibly hold. We note that the second bound $D \, (\log D)^2 \, (\log y)^k$ in Lemma \ref{sumdD} remains larger. Hence, just like before, \eqref{t3mp1} in this case will be at most:
$$
\ll \fr{q}{\vp(q)} \fr{U_1 R \, (\log U_1)^2 \, (\log y)^k}{\log R \, \log \fr{U_1}{U}}. 
$$
\end{proof}
\smallskip

Now, we give the proof of Theorem \ref{TI}
\subsection{Proof of Theorem \ref{TI}}
Let us first prove (a). We have
\begin{equation}
\begin{split}
\label{SI1sp}
S_{\tI, 1}(\al; x) &= \sum_m h(m) \sum_{mn \leq x} (\log n) \, e(mn \al) \\
&= \sum_{m \cg 0 \mm{q}} h(m) \sum_{mn \leq x} (\log n) \, e(mn \al) \, + \, \sum\lm_{\st{ m \\ q \nmid m }} h(m) \sum_{mn \leq x} (\log n) \, e(mn \al)\\ 
&= S_{\tI, 1}^{'}(\al; x) \, + \, S_{\tI, 1}^{''}(\al; x).
\end{split}
\end{equation}

Now, note that $m \leq U_1R \leq \fr{x}{8\dl_0}$ from \eqref{CI}. Therefore, applying Proposition \ref{hqm} (with $k = 1$, $y = x$, $q_0 = q$ and $f = h$) followed by an application of Lemma \ref{qhm} (with $q_0 = q$, $y = x$ and $\rho = \eta/2$), the first part $S_{\tI, 1}'$ is at most\footnote{The condition $U \geq (qR)^{1 + \eta/2}$, $U < x$ of Lemma \ref{qhm} (a) and (b) is true from \eqref{CI}.}
\begin{equation}
\label{SI1p}
\begin{split}
|S_{\tI,1}^{'}(\al; x)| &\leq \fr{x}{\dl_0} \lf( \lf| \sum_{m \cg 0 \mm{q}} \fr{h(m)}{m} \log \fr{x}{m} \rt| + \log x \lf| \sum_{m \cg 0 \mm{q}} \fr{h(m)}{m} \rt| \rt)  + O\lf( \log x \sum_{m \cg 0 \mm{q}}|h(m)| \rt) \\
&\leq \fr{x}{\dl_0 \vp(q)} \lf( 1 + O\biggl( \fr{\eta^{-1} \log x}{\log \fr{U_1}{U} \, \log qR }\biggr) \rt)
+ O\lf(  \fr{U_1 R \, \log x}{\vp(q) \, \log R \, \log \fr{U_1}{U}} \rt),
\end{split}
\end{equation}

Also, by Corollary \ref{hqnm} with $y =x$, $q_0 = q$ and $k = 1$, the second part $S_{\tI, 1}''$ is at most\footnote{The conditions $U_1 R \leq \fr{x}{8\dl_0 q}$ and $q \leq Q/R$ of Corollary \ref{hqnm} are both true from \eqref{CI}.}
\begin{equation}
\label{SI1pp}
S_{\tI,1}^{''}(\al; x) 
\ll \fr{q}{\vp(q)} \fr{U_1 R \, \log eq \, \log x}{\log R \, \log \fr{U_1}{U}} .
\end{equation}
%
Combining \eqref{SI1p} and \eqref{SI1pp} and noting that $U_1R \leq \fr{x^{1 - \eta/3}}{8 \dl_0^2 q}$ (follows from \eqref{CI}), we prove (a). 

\medskip

Now, we prove (b) and (c) simultaneously. Let 
$$f_0  = \Lb \ \, \text{or} \ \,  \mu. $$ 
We can write (as in the proof of \cite[Lemma 11.13]{HH})
\begin{equation}
\label{qxl}
\ell\al = \ell a/q + \ell \dl/x = a_1/q_\ell + \dl/x_\ell, \quad \text{where } \ q_\ell = q/(q,\ell), \, \ (a_1, q_\ell) = 1, \, \text{ and } \, x_\ell = x/\ell . 
\end{equation}
Also, we see that 
\begin{equation} 
\label{qxlc}
|\dl|/x_\ell \leq  1/q (Q/V) \quad \text{and} \quad q_\ell \leq q \leq Q/(VR) < Q/V,
\end{equation}
from \eqref{CI}. Therefore, $\ell \al$ satisfies \eqref{AP0} with $q_0 = q_\ell$, $y = x_\ell$ and $Q_0 = Q/V$. 

\smallskip

Now, we express:
\begin{equation}
\begin{split}
\label{SI2sp}
S_{\tI, 2, f_0}(\al; x) &= \sum\lm_{\ell \leq V} f_0(\ell) \sum_{m} h(m) \sum_{mn \leq x_\ell} e(mn(\ell \al)) = S_{\tI, 2, f_0}^{'}(\al; x) + S_{\tI, 2, f_0}^{''}(\al; x),
\end{split}
\end{equation}
where $S_{\tI, 2; f_0}^{'}$ is over those $m$'s such that $q_{\ell} \mid m$ and $S_{\tI, 2, f_0}^{''}$ runs over those $m$'s for which $q_{\ell} \nmid m$

Let us first consider $S_{\tI, 2, f_0}^{'}$. We see that $m \ell \leq U_1 V R \leq \fr{x}{8 \dl_0}$ (from \eqref{CI}), and hence $m \leq \fr{x_\ell}{8\dl_0}$. 
For a given $\ell$, we use Proposition \ref{hqm} to the inner sum with $k=0$, $f = h$, $q_0 = q_\ell$, $y = x_\ell$ and $\ell \al$ in place of $\al$, to obtain
\begin{equation}
\begin{split}
\label{SI21}
|S_{\tI, 2, f_0}^{'}(\al; x)| &\leq \sum_{\ell \leq V} |f_0(\ell)| \lf( \fr{x_\ell}{\dl_0} \lf| \sum_{m \cg 0 \mm{q_\ell}} \fr{h(m)}{m} \rt| \, + \, O\biggl(\log^2 x_\ell \sum_{m \cg 0 \mm{q_\ell}} |h(m)| \biggr) \rt)\\
&\leq \fr{x}{\dl_0} \sum_{\ell \leq V} \fr{|f_0(\ell)|}{\ell} \cdot \fr{2}{\vp(q_\ell) \log \fr{U_1}{U} \, \log \fr{U}{q_\ell R}} \, + \, O\lf( \fr{ U_1 R \, \log^2 x}{\vp(q) \, \log R \, \log\fr{U_1}{U}} \sum_{\ell \leq V} |f_0(\ell)| (\ell, q) \rt)\\
&\leq \fr{2x}{\dl_0 \vp(q) \, \log \fr{U_1}{U} \, \log \fr{U}{qR}} \, \sum_{\ell \leq V} \fr{|f_0(\ell)|}{\ell} (\ell, q) \, + \, O\lf( \fr{ U_1R \, \log^2 x}{\vp(q) \log R \, \log \fr{U_1}{U}} \sum_{\ell \leq V} |f_0(\ell)| (\ell, q) \rt)\\
\end{split}
\end{equation}
where we apply Lemma \ref{qhm} in the second line and use $\fr{1}{\vp(q_\ell)} \leq \fr{(\ell, q)}{\vp(q)}$ in the last line. We will later apply Lemma \ref{LbV} to the two inner sums in the last line of \eqref{SI21}, with $f_0 = \Lb \text{ or } \mu$. 

For $S_{\tI, 2, f_0}^{''}(\al; x)$, we have 
\begin{equation}
S_{\tI, 2, f_0}^{''}(\al; x) = \sum\lm_{\ell \leq V} f_0(\ell) \sum\lm_{\st{ m \\ q_{\ell} \nmid m }} h(m) \sum_{mn \leq x_\ell} e(mn(\ell \al)) .
\end{equation}
%
%
%
%
%
%
%
%
%
Now, we saw that by \eqref{qxl} and \eqref{qxlc}, $\ell \al$ satisfies \eqref{AP0} with $q_0 = q_\ell$, $y = x_\ell$ and $Q_0 = Q/V$, plus in addition we have $q_0 = q_\ell \leq Q/(VR) = Q_0/R$. Therefore, we can apply Corollary \ref{hqnm} to the inner sum over $m$ to obtain:
\begin{equation}
\begin{split}
\label{SI22}
|S_{\tI, 2, f_0}''(\al; x)| &\ll \fr{q}{\vp(q)} \fr{U_1 R \, \log^2 x}{\log R \, \log \fr{U_1}{U}} \sum_{\ell \leq V} |f_0(\ell)|
\ll \fr{q}{\vp(q)} \fr{U_1 V R \, \log^2 x}{\log R \, \log \fr{U_1}{U}} 
\end{split}
\end{equation}
%
%
since $\sum_{\ell \leq V} |f_0(\ell)| \ll V$, for $f_0 = \Lb \text{ or }\mu$.

Combining 
\eqref{SI2sp}, \eqref{SI21} and \eqref{SI22}, and finally using Lemma \ref{LbV} to the expression \eqref{SI21} in both the cases when $f_0 = \Lb \text{ or } \mu$, we find that
%
\begin{align*}
|S_{\tI, 2, \Lb}(\al; x)| &\leq \fr{2 \,x \, \log Vq}{\dl_0 \vp(q) \, \log \fr{U_1}{U} \, \log \fr{U}{qR}} \lf(1 + O\Bigl( \fr{1}{\log V} \Bigr)  \rt) \, + \, O\lf( \fr{q}{\vp(q)} \fr{U_1 V R \, \log^2 x}{\log R \, \log \fr{U_1}{U}} \rt), \\
|S_{\tI, 2, \mu}(\al; x)| &\leq \fr{2 \, x \, \tau(q) \, \log V}{\dl_0 \vp(q) \, \log \fr{U_1}{U} \, \log \fr{U}{qR}}  \lf(1 + O\Bigl( \fr{1}{\log V} \Bigr)  \rt)  \, + \, O\lf( \fr{q}{\vp(q)} \fr{U_1 V R \, \log^2 x}{\log R \, \log \fr{U_1}{U}} \rt) ,
\end{align*}
%
completing the proof of Theorem \ref{TI}. 

\bigskip

\section{Type-II sums}
\label{TIIS}
In this section, we bound the type-II sum arising from the identities \eqref{VI*} and \eqref{Vmu*}. 

For any arithmetic function $f$, let:
\begin{equation}
\label{SIIf}
S_{\tII, f}(\al; x) := \sum_{m > V} f(m) \sum_{n \leq x/m} (1*\tht)(n) \, (1 * \lb)(n) \, e(mn\al) .
\end{equation} 

Our main result is the following: 
\begin{thmI}
\label{TII}
Assume that the following conditions hold:
\begin{equation}
\label{CII}
\tag{$\mc{C}2$}
V \geq x^{\eta/3} \, \dl_0 q, \quad \ U \geq 9R^2\dl_0q, \quad \ UV < x/9.   
\end{equation}

Then for $x \geq x_0$ sufficiently large, we have:
\begin{equation*}
\begin{split}
\label{SLb}
\lf| S_{\tII, \Lb}(\al; x) \rt| 
\leq \fr{q}{\vp(q)} \fr{3.6 \, x \, \log x }{ \sr{\dl_0 q \, \log R \, \log \fr{U_1}{U}}} \int_{\fr{\log V}{\log x}}^{\fr{\log x/U}{\log x}} \sr{\fr{t}{t - \fr{\log \dl_0 q}{\log x}}} \, dt ,
\end{split}
\end{equation*}
and
\begin{equation*}
\label{Smu}
\lf| S_{\tII, \mu}(\al; x)\rt| 
\leq \fr{2 \, x}{\sr{\dl_0 \vp(q)}}  \fr{\log \fr{x}{UV}}{\sr{\log R \, \log \fr{U_1}{U}}}   .
\end{equation*}
\end{thmI}

\medskip

\subsection{Some preliminary lemmas}
The key fact, on which our argument rests upon, is this interesting connection between the Selberg sieve weights $\lb(d)$ and the Ramanujan sums $c_r(n)$. It was first observed by Kobayashi \cite{Kob} and also by Huxley \cite{Hux}. 
\begin{lem}
\label{lbcr}
Let $\{\lb(d)\}$ be the Selberg type weights in \eqref{lbR}, supported on $d \leq R$, $(d,q) = 1$. Then, for all $n \geq 1$
$$
G_q(R) \sum\lm_{d \mid n} \lb(d) = \sum\lm_{\st{r \leq R \\ (r,q) = 1 }} \fr{\mu(r) c_r(n)}{\vp(r)},
$$
where $c_r(n)$ denotes the Ramanujan sum.
\end{lem}
\begin{proof}
We have
\begin{equation*}
\begin{split}
G_q(R) \sum_{d \mid n} \lb(d) &= \sum\lm_{\st{d \mid n \\ (d,q) = 1}} d \mu(d) 
\sum\lm_{\st{r \leq R \\ r \cg 0 \mm{d} \\(r,q) = 1}} \fr{\mu^2(r)}{\vp(r)} 
 = \sum\lm_{\st{r \leq R \\(r,q) = 1}} \fr{\mu(r)}{\vp(r)} \sum\lm_{\st{d \mid n \\ d \mid r}} d \mu(r/d) 
 = \sum\lm_{\st{r \leq R \\(r,q) = 1}} \fr{\mu(r) c_r(n)} {\vp(r)} .
\end{split}
\end{equation*}
\end{proof}
The next ingredient is this simple combinatorial lemma from the author's thesis \cite{PS}. It is based upon the Brun-Titchmarsh theorem, and will play a key role in saving a logarithm. We include the proof for the sake of completeness.
\begin{lem}[\text{\cite[Lemma 3.13]{PS}}]
\label{BM}
Suppose $M  \geq 2$. We then have the following:
\begin{itemize}
\item[(a)] Let $3 \leq L \leq M/q$, and set 
$$\mf{B} = \mf{B}(L, q) = \fr{q}{\vp(q)}\fr{2 L}{\log L} .$$
Then the primes in $(M, 2M]$ can be partitioned into at most 
$\lceil \mf{B} \rceil$ subsets $\ms{M}_1, \ldots, \ms{M}_{\ceil{\mf{B}}}$, such that distinct primes congruent modulo $q$ in any $\ms{M}_j$ are separated by at least $Lq$, i.e.,
$$p \equiv p' \md{q} \, \text{ with } \, p \neq p' \, \text{ in } \, \ms{M}_j \ \Rightarrow \ \ |p-p'| \geq Lq .$$ 

\vskip 0.05in

\item[(b)]
For  $2 \leq L \leq M/q$, the integers in $(M, 2M]$ can be partitioned into at most $\ceil{L}$ subsets $\ms{M}_1, \ldots, \ms{M}_{\ceil{L}}$, such that in any $\ms{M}_j$, $m \cg m' \md{q}$ with $m \neq m'$ implies $|m - m'| \geq Lq$. 
\end{itemize}
\end{lem}
\begin{proof}    
Let us prove (a). Let $\ms{M}_j$'s be empty sets to begin with. We do the following:
\begin{itemize}
\item Let $(a, q) = 1$ and enumerate the primes in $(M, 2M]$, congruent to $a \md{q}$ in the increasing order as 
$\{ p_1(a), p_2(a), \ldots , p_K(a) \}$. Now, suppose $j \cg j_0 \md{ \lceil \mf{B} \rceil}$, where $j_0 \in \{1, 2, \ldots, \lceil \mf{B} \rceil\}$ and place $p_j(a)$ in $\ms{M}_{j_0}$. So, we place $p_1(a)$ in $\ms{M}_1$, $p_2(a)$ in $\ms{M}_2$, etc, and then place $p_{\lceil B \rceil + 1}$ again in $\ms{M}_1$ and continue this cyclically.
\vskip 0.025in
\item Repeat the above for all coprime residue classes $a \md{q}$.
\end{itemize}

$\textbf{Claim}$: If $p$ and $p'$ are distinct primes in $\ms{M}_j$ with $p \equiv p' \md{q}$, then 
$|p - p'| \geq Lq$. 

To prove this, let $p \neq p' \in \ms{M}_j$ satisfy $p \equiv p' \equiv a \md{q}$. It follows by construction that 
$p = p_i(a)$ and $p' = p_{i'}(a)$, with $i \equiv i' \md{\lceil \mf{B} \rceil}$. Since $i \neq i'$, one has 
$|i - i'| \geq \ceil{\mf{B}}$ and hence there are at least $\lceil \mf{B} \rceil$ primes $\equiv a \md{q}$ in the interval $J = (p, p']$. Let $|p - p'| = Hq$. By the Brun-Titchmarsh theorem, there are at most 
$\fr{2Hq}{\vp(q) \log H}$ primes in $J$, congruent to $a \md{q}$. Therefore
$$
\fr{2Hq}{\varphi(q) \log H} \geq \mf{B} = \frac{2 L q}{\vp(q) \, \log L},
$$
from which it follows that $H \geq L$ (as $x/\log x$ is increasing for $x \geq e$) and so $|p-p'| = Hq \geq Lq$. 
This proves the claim and (a).

For (b), we follow the exact same procedure as we followed for primes, i.e., by letting $\{\ms{M}_j: 1 \leq j \leq \ceil{L}\}$ be empty sets to begin with and given any $a \md{q}$ (not necessarily coprime to $q$), we label the integers $\cg a \md{q}$ in $(M, 2M]$ in the increasing order as $\{m_1(a), \ldots, m_K(a)\}$ and place $m_j(a)$ in $\ms{M}_{j_0}$, where $j \cg j_0 \md{\ceil{L}}$, $j_0 \in \{1, \ldots, \ceil{L}\}$. Then just like before, if $m \cg m' \cg a \md{q}$, it would follow by construction that $m = m_i(a)$, $m' = m_{i'}(a)$ with $|i - i'| \geq \ceil{L}$, implying that there are at least $\ceil{L}$ integers $\cg a \md{q}$ in $(m, m']$. And if $|m-m'|  = Hq$, one has exactly $H$ integers $\cg a \md{q}$ in $(m, m']$ and it follows that $H \geq \ceil{L} \geq L$, completing the proof. 
\end{proof}

\begin{lem}
\label{Lbmu}
For any $\eps>0$ and $M$ sufficiently large, we have
$$
\sum_{m \sim M} \Lb^2(m) \leq (1 + \eps) \, M \log M \quad \text{ and } \quad \sum_{m \sim M} \mu^2(m) \leq (6/\pi^2 + \eps) \, M .
$$
\end{lem}
\smallskip

We recall the large sieve inequality:
\begin{lem}[Large sieve inequality]
\label{LSI}
Let $\{\al_r\}_{r \in \mc{R}}$ be a set of $\dl$-spaced points in $\mb{R}/\mb{Z}$, i.e., $\|\al_i - \al_j\| \geq \dl$ for all $i \neq j$. Let $\mc{N}$ be a interval of length at most $N$. Then
$$
\sum_{r \in \mc{R}} \lf| \sum_{n \in \mc{N}}a_n e(n\al_r) \rt|^2 \leq \bigl( N + \dl^{-1} \bigr) \sum_{n \in \mc{N}} |a_n|^2 .
$$
\end{lem}
\begin{proof}
See \cite[Theorem 7.7]{IK} or \cite[Theorem 1]{MV73}. 
\end{proof}

\smallskip

Now, we arrive at the key result of this section:
\begin{prop}
\label{fII}
Let $f: \mb{N} \to \mb{C}$ be an arithmetic function. Suppose that $1 < R^2 \leq \fr{U}{8\dl_0 q}$ and that $\fr{M}{\dl_0 q} \geq \og(x)$, where $\lim_{x \to \infty} \og(x) = \infty$. Let $\mc{M} \subseteq (M, 2M]$ be a subset of integers and assume that the support of $f$ in $\mc{M}$ can be partitioned into at most $\mf{B}_f(M)$ subsets, such that in any subset, distinct elements congruent modulo $q$ are separated by at least $M/\dl_0$, i.e.,
\begin{equation} 
\label{H0}
\tag{$\mc{H}_0$}
\begin{split}
&\text{supp} (f) \cap \mc{M} = \bigsqcup_{j=1}^{\mf{B}_f(M)} \ms{M}_j, \  \text{ such that } \ m \cg m ' \md{q}, \ \, m \neq m' \text{ in } \ms{M}_j \ \Rightarrow |m - m'| \geq M/\dl_0. 
\end{split}
\end{equation}
Then, for $x \geq x_0$:
\begin{equation*}
 \lf| \sum_{m \in \mc{M}} f(m)  \sum_{n \leq x/m} (1 * \tht)(n) \, (1 * \lb)(n) \, e(mn\al) \rt| \leq (1 + \eps) \, \fr{c_0 \, x}{M}  \sr{\fr{q}{\vp(q)} \fr{\mf{B}_f(M)}{\log R \, \log \fr{U_1}{U}} \sum_{m \in \mc{M}} |f(m)|^2 },
\end{equation*} 
where $c_0 = 1 + \fr{2^{-7/3}}{(2^{1/3} - 1)} = 1.763 \ldots$. 
\end{prop}
\begin{proof}
Denote the given sum by $S_{\tII, f}(\al, \mc{M}; x)$ and let $\mc{M}_f := \text{supp}(f) \cap \mc{M}$. It is convenient to partition the $n$-sum dyadically into intervals\footnote{We partition them as $(x/2M, x/M] \cup (x/4M, x/2M] \cup \ldots$.} $\mc{N} \subseteq (N, 2N]$, with $U < N \leq x/M$ (we consider $N > U$ since $(1 * \tht)(n)$ vanishes for $n \leq U$). 

Let
\begin{equation}
\label{Sfdef}
S_{\tII, f}(\alpha, \mc{M}, \mc{N}) := \sum\lm_{\st{m \in \mc{M}}} f(m) \sum_{n \in \mc{N}}  (1* \tht)(n) \, (1 * \lb) (n) \, e(mn\al),
\end{equation}
so that $|S_{\tII, f}(\al,  \mc{M}; x)| \leq \sum\lm_{\st{ \mc{N} \text{ dyadic}\\ \mc{N} \subseteq (U, x/M]}} |S_{\tII, f}(\al, \mc{M}, \mc{N})|$. 

It is therefore enough to consider $S_{\tII, f}(\al, \mc{M}, \mc{N})$.  Applying Cauchy-Schwarz, we obtain:
\begin{equation}
\begin{split}
\label{Sftmp1}
|S_{\tII, f}(\al,  \mc{M},  \mc{N})| &\leq \lf( \sum_{m \in \mc{M}} |f(m)|^2 \rt)^{1/2}  \lf( \sum_{m \in \mc{M}_f} \left| \sum_{n \in \mc{N}}  (1*\tht)(n) \, (1 * \lb)(n) \,  e(mn\al) \rt|^2 \rt)^{1/2} \\
\end{split}
\end{equation}

Therefore, let us consider
\begin{equation*}
\begin{split}
S_{21}(m\al, \mc{N}) = \sum_{n \in \mc{N}} (1 * \tht)(n) \, (1 * \lb)(n)  \, e(mn\al) .
\end{split}
\end{equation*}

Decomposing $(1 * \lb)(n)$ via Lemma \ref{lbcr}, $S_{21}$ can be rewritten as:
\begin{equation*}
\begin{split}
G_q(R) \ S_{21}(m\al, \mc{N})
&= \sum\lm_{\st{r \leq R \\ (r,q) = 1 }} \fr{\mu(r)}{\vp(r)} \sum_{n \in \mc{N}} (1 * \tht)(n) \, c_r(n) \, e(mn \al)\\
&= \sum\lm_{\st{r \leq R \\ (r,q) = 1 }}  \fr{\mu(r)}{\vp(r)}  \sd{}{^*}\sum_{a' \md{r}} \sum_{n \in \mc{N}} (1 * \tht)(n) \,  e\lf( n (m \al + a'/r) \rt).
\end{split}
\end{equation*}

Again applying Cauchy-Schwarz, we get
\begin{equation*}
\begin{split}
G_q(R)^2 \lf| S_{21}(m\al, \mc{N}) \rt|^2
&\leq  \lf( \sum\lm_{\st{r \leq R \\ (r,q) = 1 }}  \fr{\mu^2(r)}{\vp^2(r)} \sd{}{^*}\sum_{a' \md{r}} 1 \rt) \lf( \sum\lm_{\st{r \leq R \\ (r,q) = 1 }} \ \sd{}{^*}\sum_{a' \md{r}} \lf| \sum_{n \in \mc{N}} (1 * \tht)(n) \,  e\lf( n (m \al + a'/r) \rt)  \rt|^2 \rt),
\end{split}
\end{equation*}
or
\begin{equation*}
\begin{split}
G_q(R) \lf| S_{21}(m\al,  \mc{N}) \rt|^2
&\leq \sum\lm_{\st{r \leq R \\ (r,q) = 1 }} \ \sd{}{^*}\sum_{a' \md{r}} \lf| \sum_{n \in \mc{N}} (1 * \tht)(n) \,  e\lf( n (m \al + a'/r) \rt)  \rt|^2.
\end{split}
\end{equation*}

Summing over $m$, we now obtain:
\begin{equation}
\begin{split}
\label{mstmp}
G_q(R) \sum_{m \in \mc{M}_f} \lf| S_{21}(m \al,  \mc{N}) \rt|^2
&\leq \sum_{m \in \mc{M}_f} \sum\lm_{\st{r \leq R \\ (r,q) = 1 }} \ \sd{}{^*}\sum_{a' \md{r}} \lf| \sum_{n \in \mc{N}} (1 * \tht)(n) \, \, e\lf( n (m \al + a'/r) \rt)  \rt|^2.
\end{split}
\end{equation}
%

As is standard, we will use the large sieve inequality (Lemma \ref{LSI}). To be able to apply the large sieve, we want that the quantities 
$$
\{\al_{m, a'/r} = m\al + a'/r: m \in \mc{M}_f, \, r \leq R,  \,(r, q) = 1, \ a' \md{r}, \, (a', r) = 1 \} .
$$
are well-spaced in $\mb{R}/\mb{Z}$. 

We use the assumption \eqref{H0} that $\mc{M}_f = \bigsqcup_{j=1}^{\mf{B}_f(M)} \ms{M}_j$, such that whenever $m \cg m' \md{q}$ with $m \neq m' \in \ms{M}_j$, it implies $|m-m'| \geq M/\dl_0$. 
%
When $\dl_0 = 1$ (i.e., $|\dl| \leq 4$), then $M/\dl_0$ equals $M$, and therefore the case $m \cg m' \md{q}$ with $m \neq m'$ cannot occur in any $ \ms{M}_j$ (since it would imply $|m - m'| \geq M$). Therefore, the separation between any two distinct elements $\al_{m, a'/r}$ and $\al_{m', a''/r'}$, with $m, \, m' \in \ms{M}_j$ and all fractions $a'/r, \, a''/r'$ is: 
\begin{equation*}
\begin{split}
\label{sep}
\lf\|(m-m') \bigl( a/q + \dl/x \bigr) + \fr{a'}{r} - \fr{a''}{r'} \rt\| 
&\geq \begin{cases} \fr{1}{qrr'} - \fr{M|\dl|}{x} , & |\dl| \leq 4,\\  \fr{1}{qrr'} - \fr{M|\dl|}{x}, & |\dl|>4, \ m \not\cg m' \md{q}, \\ \min\lf( \fr{M}{\dl_0} \fr{|\dl|}{x}, \fr{1}{rr'} - \fr{M|\dl|}{x} \rt), & |\dl|>4, \ m \cg m' \md{q}.  \end{cases}\\
&\geq \min \lf(\fr{4M}{x}, \fr{1}{qR^2} - \fr{M |\dl|}{x} \rt)\geq \fr{4M}{x}, 
\end{split}
\end{equation*}
since $R$ satisfies
\begin{equation}
\label{CR}
R^2 \leq \fr{U}{8 \dl_0 q}< \fr{x}{8 M \dl_0 q},
\end{equation}
%
%
Therefore, by the large sieve inequality (Lemma \ref{LSI}) applied to each $\ms{M}_j$, \eqref{mstmp} becomes
\begin{equation*}
G_q(R) \sum_{m \in \mc{M}_f} \lf| S_{21}(m \al, \mc{N}) \rt|^2
\leq \mf{B}_f(M) \lf( |\mc{N}| + \fr{x}{4M} \rt) \sum_{n \in \mc{N}} (1 * \tht)(n)^2 .
\end{equation*}
Note that $\tht'$ corresponds to the Barban-Vehov weights, and $1 * \tht = 1 * \mu - (1 * \tht') = -(1 * \tht')$ (since the parameter $n \gg 1$ in our range). We recall the following result of Graham \cite{Gr2}:
\begin{equation} 
\begin{split}
\label{1*tht}
\sum_{n \leq X} (1 * \tht')(n)^2  &=  \fr{X \, \log \fr{\min(X, U_1)}{U} }{(\log U_1/U)^2} \, + \, O\lf( \fr{X}{\log^2 U_1/U} \rt),
\end{split}
\end{equation}
for $X > U$. The above result implies that (since $\mc{N}$ is a dyadic interval in $(U, x/M]$):
\begin{equation}
\sum_{n \in \mc{N}} (1 * \tht)(n)^2  \leq \fr{(1 + \eps) \, |\mc{N}|}{ \log \fr{U_1}{U}}.
\end{equation}
This effectively yields (using $G_q(R) \geq  \fr{\vp(q)}{q} \log R$):
\begin{equation}
\label{Mftmp1}
\sum_{m \in \mc{M}_f} \lf| S_{21}(m \al, \mc{N}) \rt|^2
\leq (1 + \eps) \, \fr{q}{\vp(q)} \fr{\mf{B}_f(M) \, |\mc{N}|^2}{\log R \, \log \fr{U_1}{U}} \lf( 1 + \fr{x}{4M |\mc{N}|} \rt)  .
\end{equation}

Now, we plug \eqref{Mftmp1} in \eqref{Sftmp1}, sum over $\mc{N}$ and use the inequality\footnote{The reason for avoiding the use of $\sr{1+z} \leq 1 + \sr{z}$, is that it is weaker from a numerical point of view.} $\sr{1+z} \leq 1 + \fr{1}{2} z^{2/3}$,  for all $z \geq 0$ to obtain:
\begin{equation*}
\begin{split}
|S_{\tII, f}(\al, \mc{M}; x)| &\leq (1 + \eps) \sr{\fr{q}{\vp(q)} \fr{\mf{B}_f(M)}{\log R \, \log \fr{U_1}{U}}  \sum_{m \in \mc{M}} |f(m)|^2 } \sum\lm_{\st{\mc{N} \text{ dyadic} \\ \mc{N} \subseteq (U, x/M]}} 
\lf( |\mc{N}| +  \fr{1}{2} \Bigl( \fr{x}{4M} \Bigr)^{2/3}|\mc{N}|^{1/3} \rt)\\
&\leq (1 + \eps) \fr{c_0 \,  x}{M}\sr{\fr{q}{\vp(q)} \fr{\mf{B}_f(M)}{\log R \, \log \fr{U_1}{U}} \sum_{m \in \mc{M}} |f(m)|^2},
\end{split}
\end{equation*}
%
with $c_0 = 1 + \fr{2^{-7/3}}{2^{1/3} - 1}$,  since $\sum\lm_{\st{\mc{N} \text{ dyadic} \\ \mc{N} \subseteq (U, x/M]}}  |\mc{N}|^{\vartheta} \leq \fr{1}{2^{\vartheta} - 1}(x/M)^{\vartheta}$, for all $0 < \vartheta \leq 1$.   
%
%
\end{proof}
\begin{rem}
In the above proof, the $(1 * \lb)$ essentially mimics the effect of an averaged Montgomery's inequality (thanks to Lemma \ref{lbcr}). The $L^2$-norm of $(1 * \tht)$ gives yet another logarithmic saving. Moreover, in the case when $f = \Lb$, even $B_{\Lb}(M)$ saves yet another $\log$ (Lemma \ref{BM} (a)). The proof draws upon ideas from \cite[Thm 3.3]{PS} of the author's thesis, where the sum $\sum_{p_1 \in \mc{M}}\sum_{p_2 \in \mc{N}} a(p_1) b(p_2) \, e(p_1 p_2 \al) $, with $p_1$, $p_2$ primes is considered. 
\end{rem}

\smallskip

The following is a straightforward consequence of the Euler summation formula.
\begin{lem}
\label{dyadic}
Suppose that $1 < A < B < x$ and let $\phi: (0,\infty) \to [0, \infty)$ be such that $\phi'(t) \leq 0$. Then, we have:
\begin{equation*}
\sum\lm_{\st{A< m \leq B \\ m = 2^k}} \phi\lf( \fr{\log m}{\log x} \rt) \leq \fr{\log x}{\log 2} \int_{\fr{\log A}{\log x}}^{\fr{\log B}{\log x}} \phi(t) \, dt \, + \, \phi\lf( \fr{\log A}{\log x} \rt) .
\end{equation*}
\end{lem}
\medskip

\subsection{Proof of Theorem \ref{TII}}
Let us first consider the sum with $\Lb$. We split 
$$\Lb(m) = \bm{1}_{\mb{P}}(m) \log m  + \Lb'(m), $$ 
where $\bm{1}_{\mb{P}}$ is the indicator function of the primes and $\Lb'(m) = \begin{cases} \log p, & m = p^k, \text{ with } k \geq 2 \\0, & \text{otherwise.} \end{cases}$. 

We deal with the first term $\bm{1}_{\mb{P}} \cdot \log$ first. Applying Lemma \ref{BM} (a) with $L = \fr{M}{\dl_0 q}$, we can split the primes in $(M, 2M]$ into at most $\ceil{\mf{B}_1(M)}$ subsets 
$\ms{M}_1, \ldots, \ms{M}_{\ceil{\mf{B}_1}}$, where 
$$\mf{B}_{1}(M) = \fr{q}{\vp(q)} \fr{2L}{\log L} = \fr{2q}{\vp(q)} \fr{M}{\dl_0 q \, \log \fr{M}{\dl_0 q}},
$$
such that the hypothesis \eqref{H0} of Proposition \ref{fII} is satisfied. We also see that since $M/(\dl_0 q)$ is sufficiently large, one has $\ceil{\mf{B}_1} \leq (1 + \eps) \mf{B}_1$. 

We also have from Lemma \ref{Lbmu}, that $\sum_{p \sim M} (\log p)^2 \leq (1 + \eps) M \log M$. By an application of Proposition \ref{fII} with $f = \bm{1}_{\mb{P}} \cdot \log$, we have (using the notation of $S_{\tII, f}(\al, \mc{M}; x)$ from \eqref{Sfdef})
\begin{equation}
\begin{split}
\label{1Plog}
|S_{\tII, \bm{1}_{\mb{P}} \cdot \log}(\al, \mc{M}; x)| 
&\leq \fr{(c_0 + \eps) \, x}{M} \sr{\fr{q}{\vp(q)} \fr{\ceil{\mf{B}_1}}{\log R \, \log \fr{U_1}{U}} \, \sum_{p \sim M} (\log p)^2 } \\
&\leq \fr{q}{\vp(q)} \fr{(c_0 \, \sr{2} + \eps) \, x}{\sr{\dl_0 q}} \sr{\fr{\log M}{ \log R \, \log \fr{M}{\dl_0 q} \, \log \fr{U_1}{U} }}.
\end{split}
\end{equation}

\smallskip

Next, we shall estimate the sum over $\Lb'$. Note that
\begin{equation*}
\begin{split}
\sum_{m \in \mc{M}} \Lb'(m) &\leq \sum_{k=2}^{\floor{\log_2(2M)}} \sum_{p^k \sim M} \log p 
\leq (1 + \eps) \sum_{k=2}^{\floor{\log_2(2M)}} M^{1/k} 
\leq (1 + \eps) \sr{M},
\end{split}
\end{equation*}
since $M$ is sufficiently large. 

Also, the $L^2$-norm of the Selberg-sieve weights \eqref{lbR} is just the upper bound the sieve yields (see \cite{RH})
\begin{equation} 
\label{1*lb}
\sum_{n \leq X} (1 * \lb) (n)^2 = \fr{X}{G_q(R)} + O\lf( \fr{R^2}{G_q(R)^2} \rt) \leq \fr{(1 + \eps) \, X}{G_q(R)},
\end{equation}
whenever $X \gg R^2$. 
Therefore, bounding the sum trivially, we have (using \eqref{1*tht} and \eqref{1*lb}), that\footnote{We can apply \eqref{1*tht} and \eqref{1*lb}, since $x/M >  U \gg R^2$.}
\begin{equation}
\begin{split}
\label{Lbp}
|S_{\tII, \Lb'}(\al, \mc{M}; x)| &\leq \lf( \sum_{m \in \mc{M}} \Lb'(m) \rt)  \Biggl(\sum_{n \leq x/M} |(1*\tht)(n) \, (1 * \lb)(n)| \Biggr)\\ 
&\leq (1 + \eps) \sr{M} \Biggl( \sum_{n \leq x/M} (1 * \tht)(n)^2 \Biggr)^{1/2} \Biggr( \sum_{n \leq x/M} (1 * \lb)(n)^2 \Biggr)^{1/2} \\
&\leq \sr{\fr{q}{\vp(q)}} \ \fr{(1 + \eps) \, x}{\sr{M} \sr{\log R \, \log \fr{U_1}{U}}} .
\end{split}
\end{equation}
%
Combining \eqref{1Plog} and \eqref{Lbp}, and noting that in \eqref{Lbp}, $1/\sr{M} \leq \eps/\sr{\dl_0 q}$, it follows that:
$$
|S_{\tII, \Lb}(\al, \mc{M}; x)| \leq \fr{q}{\vp(q)} \fr{(c_0 \sr{2} + \eps) \, x}{\sr{\dl_0 q}} \sr{\fr{\log M}{ \log R \, \log \fr{M}{\dl_0 q} \, \log \fr{U_1}{U}}} .
$$
Now, summing over the dyadic intervals $\mc{M} \subseteq (V, x/U]$ above (using Lemma \ref{dyadic}), we get
\begin{equation*}
\begin{split}
|S_{\tII, \Lb}(\al; x)| &\leq \fr{q}{\vp(q)} \fr{(c_0\sr{2} + \eps) \, x}{\sr{\dl_0 q \, \log R \, \log \fr{U_1}{U}}} \sum\lm_{\st{V < M \leq x/U \\ M = 2^k}} \sr{\fr{\log M}{\log \fr{M}{\dl_0 q}}} \\
&\leq \biggl( \fr{c_0\sr{2}}{\log 2} + \eps \biggr) \,\fr{q}{\vp(q)} \fr{ x \, \log x}{ \sr{\dl_0 q \, \log R \, \log \fr{U_1}{U}}} \, \Biggl( \int_{\fr{\log V}{\log x}}^{\fr{\log x/U}{\log x}} \sr{\fr{t}{t - \fr{\log \dl_0 q}{\log x}}} \ dt + O\biggl( \fr{\eta^{-1}}{\log x} \biggr) \Biggr)\\
&\leq \fr{q}{\vp(q)} \fr{ 3.6 \, x \, \log x}{ \sr{\dl_0 q \, \log R \, \log \fr{U_1}{U}}} \int_{\fr{\log V}{\log x}}^{\fr{\log x/U}{\log x}} \sr{\fr{t}{t - \fr{\log \dl_0 q}{\log x}}} \ dt .
\end{split}
\end{equation*} 
%

\smallskip

Next, we consider $S_{\tII, \mu}(\al, \mc{M})$. By Lemma \ref{BM} (b) with $L = M/\dl_0 q$, we can partition $(M, 2M]$ into at most $\mf{B}_2 = \ceil{\fr{M}{\dl_0 q}}$ subsets, such that the hypothesis \eqref{H0} of Proposition \ref{fII} is satisfied. Therefore, applying Proposition \ref{fII} with $f = \mu$ and using Lemma \ref{Lbmu}, we find
\begin{equation*}
|S_{\tII, \mu}(\al, \mc{M}; x)| \leq \sr{\fr{q}{\vp(q)}} \fr{(c_0\sr{6/\pi^2} + \eps) \, x}{\sr{\dl_0 q \, \log R \, \log \fr{U_1}{U}}} .
\end{equation*}
Now, summing over dyadic intervals using Lemma \ref{dyadic}, we get:
$$
|S_{\tII, \mu}(\al; x)| \leq (1 + \eps) \fr{c_0 \sr{6}}{\pi \log 2} \, \fr{x}{\sr{\dl_0 \vp(q)}} \fr{\log \fr{x}{UV} + O(1)}{\sr{\log R \, \log \fr{U_1}{U}}} \leq \fr{2 \, x}{\sr{\dl_0 \vp(q)}}  \fr{\log \fr{x}{UV}}{\sr{\log R \, \log \fr{U_1}{U}}}  .
$$

This completes the proof of Theorem \ref{TII}. 

\medskip

\section{Completing the proof of Theorem \ref{main}}

\subsection{Combining all bounds}
First, let us evaluate the final expression for the Von-Mangoldt function. Recall from \eqref{VI*}, that we have $\Lb = h * \log - 1 *h * \Lb_{\leq V} + (1 * \tht)(1 * \lb) * \Lb_{>V} + \Lb_{\leq V}$. From Theorem \ref{TI}, we get the bounds for the two type-I sums, while from Theorem \ref{TII}, we get the type-II estimate. Now, the fourth term (arising from $\Lb_{\leq V}$) is just $O(V)$. 

Therefore, under the conditions \eqref{CI} and \eqref{CII}, one has:
\begin{equation}
\begin{split}
\label{Lbtotal}
\lf| \sum_{n \leq x} \Lb(n) e(n\al) \rt| &\leq (1 + \eps) \fr{q}{\vp(q)} \lf(\fr{x}{\dl_0 q} + \fr{3.6 \,  x \, \log x}{ \sr{\dl_0 q \, \log R \, \log \fr{U_1}{U}}} \, \int_{\fr{\log V}{\log x}}^{\fr{\log x/U}{\log x}} \sr{\fr{t}{t - \fr{\log \dl_0 q}{\log x}}} \ dt  \rt) \\
&\qquad + O\lf(\fr{q}{\vp(q)} \fr{U_1VR \, \log^2 x}{\log R \, \log \fr{U_1}{U}} \rt) .
\end{split}
\end{equation}

\smallskip

Now, we consider $\mu$. Recall from \eqref{Vmu*}, that $\mu = h - 1 * h * \mu_{\leq V} + (1*\tht)(1 * \lb) * \mu_{>V} + \mu_{\leq V}$. By Lemma \ref{qhm} (c) (with $q_0 = 1$), one can bound the estimate from the first term. For the second term, Theorem \ref{TI} (c) does the job. The last term is just $O(V)$.

Therefore, under \eqref{CI} and \eqref{CII}, one obtains (using $\tau(q) \leq 3 \sr{\vp(q)} \leq 3 \sr{\dl_0 \vp(q)}$.)
\begin{equation}
\begin{split}
\label{mufinal}
\lf| \sum_{n \leq x} \mu(n) e(n \al) \rt| &\leq \fr{9 \, x \, \log V}{\sr{\dl_0 \vp(q)} \, \log \fr{U_1}{U} \, \log \fr{U}{qR}}  + \fr{2 \, x}{\sr{\dl_0 \vp(q)}} \fr{\log \fr{x}{UV}}{\sr{\log R \, \log \fr{U_1}{U}}}\\
&\qquad + O\lf( \fr{q}{\vp(q)} \fr{U_1 V R \, \log^2 x}{\log R \, \log \fr{U_1}{U}} \rt) .
\end{split}
\end{equation}
%
%

\smallskip

\subsection{Choice of the parameters}
Now that we have obtained the final bounds for the exponential sums over $\Lb$ and $\mu$ in \eqref{Lbtotal} and \eqref{mufinal}, let us figure the optimal choice of the parameters $U$, $U_1$, $R$ and $V$. Moreover, it is seen that it is the type-II contribution, which primarily decides the asymptotic constants, as the type-I plus the error term contribute very little.

For the sake of convenience, let 
$$ U_1/U = R_1 .$$ 
We observe that there is a term $O\Bigl( \fr{q}{\vp(q)} \fr{U_1VR \, \log^2 x}{\log R \, \log R_1} \Bigr)$ occurring in both \eqref{Lbtotal} and \eqref{mufinal}, which needs to be kept in check. Therefore, we also assume that 
\begin{equation} 
\label{nC}
U_1VR = UVRR_1 \leq \fr{x}{8 \sr{\dl_0 q} \, \log^2 x} .
\end{equation} 
This ensures that the last error term in \eqref{Lbtotal} and \eqref{mufinal} is at most $O\Bigl( \fr{x \, \log \log q}{\sr{\dl_0 q} \, \log R \, \log \fr{U_1}{U}} \Bigr)$, since $q/\vp(q) \ll \log \log q$. 

Therefore, we assume the following (taking into account \eqref{CI}, \eqref{CII} and \eqref{nC}):
%
\begin{equation}
\label{C0}
\tag{$\mc{C}12$}
\begin{split}
&V \geq x^{\eta/3} \dl_0 q, \quad \, R \geq x^{\eta/4}, \quad  \, U \geq 9 \, \max\Bigl( \dl_0 (Rq)^{1 + \eta/2}, R^2\dl_0 q \Bigr), \quad \, U V R R_1 \leq  \fr{x \, \min\bigl( 1, \sr{ \fr{q}{\dl_0} }\bigr)}{8 \, \sr{\dl_0 q} \, \log^2 x}, \\ 
&\text{and } \ qVR \leq Q. 
\end{split}
\end{equation}

After some back of the envelope calculations, we make the following choice (these are not exactly optimal but close and also keep things elegant):
%
%
\begin{equation}
\label{choiceLb}
V = x^{(\eta - \eta^3)/2} \, \dl_0 q, \quad U = \lf(\fr{x^{1 - \eta/2} \, \Delta}{(\dl_0 q)^{1/2}}\rt)^{1/2} , \quad R = R_1 = \fr{1}{3}\lf( \fr{x^{1 - \eta/2} \, \Delta}{(\dl_0q)^{5/2}} \rt)^{1/4},
\end{equation}
where 
$$
\Delta = \min\lf( 1, \sr{\fr{q}{\dl_0}} \rt).
$$
Moreover, it can be readily seen that 
$$\fr{\log \Delta}{\log x} = -\fr{\log^+ \dl_0/q}{2 \log x},$$
where $z^+$ equals $\max(z, 0)$. 

\begin{rem}
The quantity $\Delta$ comes into play only when $q \leq x^{1/10 + \eta/2}$, since that is when $\dl_0$ could be possibly larger than $q$. Once $q > x^{1/10 + \eta/2}$, we have $\Delta = 1$. 
\end{rem}
\subsubsection{Verification of the conditions \tps{\eqref{C0}}{}}
We now verify that the parameters in \eqref{choiceLb} satisfy the conditions \eqref{C0}. 

First, we ensure that $R = R_1 \geq x^{\eta/4}$, i.e., $x(\dl_0 q)^{-5/2} \Delta \geq x^{\eta}$. Now, this is at least $\min\bigl(\fr{x}{(\dl_0q)^{5/2}}, \fr{x}{\dl_0^3 q^2} \bigr)$. The first one is clearly $\gg x^{5\eta/2}$ (as $\dl_0q \leq x^{2/5 - \eta}$). Also since $\dl_0 \ll x^{1/5 + \eta}/q$, the second one is at least $\gg x^{2/5 - 3\eta} q \geq x^{\eta}$, since $\eta \leq 1/10$. Thus $R = R_1 \gg x^{\eta/4}$ holds. 

Now, $V = x^{\eta/2 - \eta^3/2} \, \dl_0 q \geq x^{\eta/3} \, \dl_0 q$ clearly holds for all $\eta \leq 1/10$. 

Next, it can be easily seen that $U =  9R^2 \dl_0q$ with this choice. And $R^2 \dl_0 q \gg \dl_0(R q)^{1 + \eta/2}$ holds if and only if $q^{2\eta} (\dl_0 q)^{\fr{5(1-\eta/2)}{2}} \ll x^{(1-\eta/2)^2} \Delta^{(1-\eta/2)}$. Now, if $\Delta = 1$ (i.e. $q > \dl_0$), then this is equivalent to $\dl_0^{5/2 - 5\eta/4} q^{5/2 + 3\eta/4} \ll x^{(1-\eta/2)^2}$. Since $\dl_0q \leq x^{2/5 - \eta}$, it is enough to have $(2/5 - \eta)(5/2 + 3\eta/4) < (1 - \eta/2)^2 \iff 7\eta/4 + 6/5>0$, which is true. Now, we assume that $\Delta = \sr{q/\dl_0}$ (i.e., $\dl_0 > q$ which implies $q < x^{1/10 + \eta/2}$ and hence $\dl_0q \leq x/Q = x^{1/5 + \eta}$). Then this reduces to $\dl_0^{3 - 3\eta/2} q^{2 + \eta} \ll x^{(1-\eta/2)^2}$. Now, since $3-3\eta/2 > 2 + \eta$ and $\dl_0q \leq x^{1/5 + \eta}$, it's enough to have $(3-3\eta/2)(1/5 + \eta) < (1 -\eta/2)^2 \iff 2/5 + 7\eta^2/4 > 3.7 \eta$, which holds for $\eta \leq 1/10$. 
 
Now, $UVRR_1 = UVR^2 = \fr{x^{1 - \eta^3/2}}{9 \sr{\dl_0 q}} \Delta$, which is exactly what we want, since the $x^{\eta^3/2}$ power-saving takes care of any power of $\log x$. 

Next, $qVR \leq (\dl_0 q) V R = x^{1/4 + 3\eta/8 - \eta^3/2} (\dl_0 q)^{11/8} \leq x^{\fr{11}{8} \bigl( 2/5 - \eta \bigr) + 1/4 + 3\eta/8 - \eta^3/2} 
< x^{4/5 - \eta} = Q$. 

We have thus verified that the parameters in \eqref{choiceLb} satisfy all the conditions of \eqref{C0}. 
\smallskip
\subsection{Concluding the proof}
With the above choice of the parameters, we see that \eqref{Lbtotal} becomes
\begin{equation*}
\begin{split}
\lf| \sum_{n \leq x} \Lb(n) e(n\al) \rt| 
&\leq \fr{q}{\vp(q)} \fr{(1 + \eps) \, x}{\sr{\dl_0 q}} \lf( 1 +  \fr{14.4}{\bigl(1 - \fr{\eta}{2} - \fr{5 \log \dl_0q}{2\log x} - \fr{\log^+ \dl_0/q}{2 \log x} \bigr)} 
\int\lm_{ \fr{\eta - \eta^3}{2}  +  \fr{\log \dl_0 q}{\log x} }^{\fr{1}{2} + \fr{\eta}{4} + \fr{\log \dl_0 q}{4 \log x} + \fr{\log^+ \dl_0/q}{4 \log x}} \sr{\fr{t}{t - \fr{\log \dl_0 q}{\log x}}} \ dt \rt)\\
&\leq \fr{q}{\vp(q)} \fr{x}{\sr{\dl_0 q}} \, \ms{F}_{\eta}\lf( \fr{\log \dl_0 q}{\log x}, \fr{\log^+ \dl_0/q}{\log x} \rt),
\end{split}
\end{equation*}
with 
$$
\ms{F}_{\eta}(u, u_0) = 1.01 + \fr{14.41}{1 - (\eta +  5u + u_0)/2} \int_{\fr{\eta - \eta^3}{2} + u}^{(2 + \eta + u + u_0)/4}  \sr{\fr{t}{t-u}} \, dt,
$$
for all $0 \leq u \leq 2/5 - \eta$ and $0 \leq u_0 \leq \min\bigl( u,1/5 + \eta \bigr)$. 
\smallskip

And \eqref{mufinal} becomes:
\begin{equation}
\begin{split}
\lf| \sum_{n \leq x} \mu(n) e(n\al) \rt| &\leq \fr{(2 + \eps)  \, x}{\sr{\dl_0 \vp(q)}}\lf(\fr{\log \biggl( \fr{x^{1/2 - \eta/4 + \eta^3/2}}{(\dl_0q)^{3/4} \Delta^{1/2}}\biggr)}{ \fr{1}{4} \log \fr{x^{1 - \eta/2} \, \Delta}{(\dl_0q)^{5/2}}}  \rt) 
\leq \fr{x}{\sr{\dl_0 \vp(q)}} \, \ms{G}_{\eta}\lf( \fr{\log \dl_0 q}{\log x}, \fr{\log^+ \dl_0/q}{\log x} \rt),
\end{split}
\end{equation}
with 
$$
\ms{G}_{\eta}(u, u_0) = 4.01 \, \fr{1  + \eta^3 -(\eta + 3u + u_0)/2}{ 1 - (\eta + 5u + u_0)/2},
$$
again for $0 \leq u \leq 2/5 - \eta$ and $0 \leq u_0 \leq \min\bigl( u,1/5 + \eta \bigr)$. 

This completes the proof of Theorem \ref{main}. 
%
\begin{rem}
As we are saving enough logarithms at all stages throughout the proof, we do not require $\dl_0 q$ to be bounded from below. The trivial fact $\dl_0 q \geq 1$ is alone sufficient!
\end{rem}
\vskip 0.05in
\subsection*{Acknowledgements.} I am grateful to Prof. Harald Helfgott for many enlightening discussions, and his valuable feedback on major portions of the paper. I also thank Prof. Olivier Ramar\'{e} for many fruitful interactions in the past, and Prof. R. Balasubramanian for his guidance during my doctoral years.

\bigskip

\printbibliography

\vskip 0.36	in

\hrule

\end{document}